\IfFileExists{./prepreamble-amspreprint.sty}{\RequirePackage[packages,theorems,changes]{prepreamble-amspreprint}}{}

\makeatletter
\IfFileExists{./scoop-latex/scoop-style.sty}{\providecommand*{\input@path}{}\edef\input@path{{./scoop-latex/}\input@path}}{}
\makeatother
\PassOptionsToPackage{style = authoryear-comp, backend = biber}{biblatex}
\PassOptionsToPackage{commentmarkup = footnote}{changes}
\PassOptionsToPackage{inline}{enumitem}
\RequirePackage{algorithm}
\RequirePackage{algpseudocode}    

\documentclass[english]{amsart}

\RequirePackage{biblatex}
\RequirePackage{ltxcmds}
\IfFileExists{./preamble-amspreprint.sty}{\RequirePackage[packages,theorems,changes]{preamble-amspreprint}}{}

\usepackage{manuscript}
\usepgfplotslibrary{groupplots}

\IfFileExists{./postpreamble-amspreprint.sty}{\RequirePackage[packages,theorems,changes]{postpreamble-amspreprint}}{}

\makeatletter
\@ifpackageloaded{changes}{
\definechangesauthor[name = {Ronny Bergmann}, color = {TolVibrantMagenta}]{RB}
\definechangesauthor[name = {Hajg Jasa}, color = {TolVibrantOrange}]{HJ}
\definechangesauthor[name = {Paula John}, color = {TolBrightPurple}]{PJ}
\definechangesauthor[name = {Max Pfeffer}, color = {TolBrightBlue}]{MP}
}{}
\makeatother        

\makeatletter
\@ifpackageloaded{hyperref}{%
	\hypersetup{
		pdftitle = {The Intrinsic Riemannian Proximal Gradient Method for Nonconvex Optimization},
		pdfauthor = {Ronny Bergmann, Hajg Jasa, Paula John, Max Pfeffer},
		pdfkeywords = {splitting methods, proximal methods, Riemannian manifolds, nonconvex functions},
		pdfcreator = {Created using the Scoop Template Engine version 1.6.0.}
	}
}{
	\pdfinfo{
		/Title (The Intrinsic Riemannian Proximal Gradient Method for Nonconvex Optimization)
		/Author (Ronny Bergmann, Hajg Jasa, Paula John, Max Pfeffer)
		/Subject ()
		/Keywords (splitting methods, proximal methods, Riemannian manifolds, nonconvex functions)
		/Creator (Created using the Scoop Template Engine version 1.6.0.)
	}
}
\makeatother

\title[The Nonconvex Riemannian Proximal Gradient Method]{The Intrinsic Riemannian Proximal Gradient Method for Nonconvex Optimization}

\author[R. Bergmann]{Ronny Bergmann\orcidlink{0000-0001-8342-7218}}
\address[R. Bergmann]{Norwegian University of Science and Technology, Department of Mathematical Sciences, NO-7041 Trondheim, Norway}
\email{\detokenize{ronny.bergmann@ntnu.no}}
\urladdr{https://www.ntnu.edu/employees/ronny.bergmann}

\author[H. Jasa]{Hajg Jasa\orcidlink{0009-0002-7917-0530}}
\address[H. Jasa]{Norwegian University of Science and Technology, Department of Mathematical Sciences, NO-7041 Trondheim, Norway}
\email{\detokenize{hajg.jasa@ntnu.no}}
\urladdr{https://hajg-ijk.github.io/}

\author[P. John]{Paula John\orcidlink{0009-0009-2500-5485}}
\address[P. John]{RWTH Aachen University, Institute for Geometry and Applied Mathematics, Im Süsterfeld 2, D-52072 Aachen, Germany}
\email{\detokenize{john@igpm.rwth-aachen.de}}
\urladdr{https://www.igpm.rwth-aachen.de/team/john}

\author[M. Pfeffer]{Max Pfeffer\orcidlink{0000-0002-7739-4031}}
\address[M. Pfeffer]{Universität Potsdam, Faculty of Health Sciences Brandenburg, Am Mühlenberg 9, Building 62 (H-Lab), 14476 Potsdam – Golm, Germany}
\email{\detokenize{max.pfeffer@uni-potsdam.de}}
\urladdr{https://www.maxpfeffer.com/}

\date{\today}

\dedicatory{}

\begin{document}

% Insert the abstract.
\begin{abstract}
We consider the proximal gradient method on Riemannian manifolds for functions that are possibly not geodesically convex.
Starting from the forward-backward-splitting, we define an intrinsic variant of the proximal gradient method that uses proximal maps defined on the manifold and therefore does not require or work in any embedding.
We investigate its convergence properties and illustrate its numerical performance, particularly for nonconvex or nonembedded problems that are hence out of reach for other methods.

\end{abstract}

% Insert the keywords.
\keywords{splitting methods, proximal methods, Riemannian manifolds, nonconvex functions}

% Insert the Mathematics Subject Classification.
\makeatletter
\ltx@ifpackageloaded{hyperref}{%
\subjclass[2010]{\href{https://mathscinet.ams.org/msc/msc2020.html?t=90C26}{90C26}, \href{https://mathscinet.ams.org/msc/msc2020.html?t=49Q99}{49Q99}, \href{https://mathscinet.ams.org/msc/msc2020.html?t=49M30}{49M30}, \href{https://mathscinet.ams.org/msc/msc2020.html?t=65K10}{65K10}}
}{%
\subjclass[2010]{90C26, 49Q99, 49M30, 65K10}
}
\makeatother

% Typeset the opening page.
\maketitle

% Insert the document body.
\section{Introduction}
\label{section:introduction}

Nonsmooth optimization on Riemannian manifolds has recently gained considerable interest.
Starting with inaugural algorithms like the subgradient method~\cite{FerreiraOliveira:1998:1} and the proximal point algorithm~\cite{FerreiraOliveira:2002:1}, introducing more properties of either the objective or the manifold has led to the book \cite{Bacak:2014:2}.
For methods involving a bundle or approaches that involve a bundle of subgradients, see~\cite{HoseiniMonjeziNobakhtianPouryayevali:2021:1,BergmannHerzogJasa:2024:1}.
Especially the proximal bundle method from~\cite{HoseiniMonjeziNobakhtianPouryayevali:2021:1}
is one of the first methods to investigate the non-convex case.

A prominent approach to nonsmooth optimization are splitting methods, where the objective function
can be split. For a Riemannian manifold $\cM$, we consider the problem
\begin{equation}
    \label{eq:splitting}
    \argmin_{p \in \cM}
    f(p)
    ,
    \qquad
    \text{where}
    \quad
    f(p)
    =
    g(p)
    +
    h(p)
    .
\end{equation}
Some, like the cyclic proximal point algorithm (CPPA)~\cite{Bacak:2014:1} involve an arbitrary number of summands, while others also work for various splittings, \eg, convex-concave-splittings or difference of convex (DC) problems~\cite{SouzaOliveira2015,AlmeidaNetoOliveiraSouza2020,BergmannFerreiraSantosSouza:2024}.%

In the case with two functions, as in \cref{eq:splitting}, the goal is to exploit properties of $g$ and $h$ that allow for an efficient minimization. For example when both have “cheap” proximal maps, the Douglas-Rachford algorithm~\cite{BergmannPerschSteidl:2016:1} can be used also on Hadamard manifolds.
Another prominent algorithm is the so-called Alternating Direction Method of Multipliers (ADMM)~\cite{GabayMercier:1976:1}, for which variants that work on manifolds were,~\eg, considered in~\cite{KovnatskyGlashoffBronstein:2016,JiaxiangMaSrivastava:2022}.
Finally, together with an investigation of Fenchel duality, \cite{SilvaLouzeiroBergmannHerzog:2022:1,SchielaHerzogBergmann:2024} allows for the so-called Chambolle-Pock algorithm~\cite{BergmannHerzogSilvaLouzeiroTenbrinckVidalNunez:2021:1} to be defined on manifolds, also involving proximal maps.%

A prominent algorithm that is based on splitting is the proximal-gradient method, see~\cite[Chapter 10]{Beck:2017:1} for a comprehensive overview in the Euclidean case.
In this case, $g$ is assumed to be ($L_g$-)smooth and $h$ to be convex and potentially nonsmooth.
If $h$ has a proximal map that is fast to evaluate, maybe even in closed form, the algorithm is efficient, and a convergence rate can be derived.

To the best of our knowledge, the authors in~\cite{ChenMaMan-ChoSoZhan:2020:1} first introduced a partially convex
proximal-gradient version for the Stiefel manifold, phrasing the nonsmooth part as a constrained problem in the embedding, and requiring the nonsmooth function to be the restriction of a convex function in the embedding.
This was generalized to general Riemannian manifolds in~\cite{HuangWei:2021:1}.
In this work, the authors also phrase the nonsmooth subproblem as a problem in the tangent space at the current iterate, but forgo convexity assumptions in their general analysis.
The same authors generalize their framework to an inexact version in~\cite{HuangWei:2023:1} and still keep their convergence analysis from the previous paper.
The linear convergence of~\cite{ChenMaMan-ChoSoZhan:2020:1} and~\cite{HuangWei:2021:1} under strong convexity has been studied in~\cite{ChoiChunJungYun:2024:1}.

The authors in~\cite{ZhaoYanZhu:2023} rephrase the nonsmooth subproblem as a trust-region problem.
All these works are based on a “generalized notion of the proximal map”~\cite{HuangWei:2021:1},
and solving the corresponding subproblem either in the embedding or the tangent space of the current iterate.
Most of these works assume $g$ and $h$ fulfil some convexity property, either along geodesics or along curves generated by retractions when using a retraction-based generalization of the proximal subproblem.
In the work of~\cite{HuangWei:2021:1}, a “retraction smoothness” of $g$ is required, but $h$ can be nonsmooth and nonconvex for their convergence investigations.
However, their algorithm employs the embedding in a way that $h$ is required to have certain properties thereon in order to solve their approximation to the proximal map.
In~\cite{FengHuangSongYingZeng:2021} the authors do employ a proximal map, but restrict their investigation to Hadamard manifolds and employing the Kurdyka-Łojasiewicz property on manifolds.
In a related work,~\cite{MartinezRubioPokutta:2023:1} consider gradient schemes that combine approximate solutions to proximal maps defined on the manifold, leading to a globally accelerated gradient descent algorithm on Hadamard manifolds.

\subsection{Contributions}
In this work, we introduce a formulation of the Riemannian proximal gradient method that actually employs proximal maps for the nonsmooth step.
This relates the new proximal gradient method directly back to both the PPA~\cite{FerreiraOliveira:2002:1} as well as gradient descent.
This new algorithm can still be generalized to a corresponding retraction-based proximal map,
which, similarly to~\cite{HuangWei:2021:1}, can be understood as an approximate variant.
This new formulation allows to investigate several properties from~\cite[Chapter 10]{Beck:2017:1} on manifolds, for the case where neither the smooth part $g$ nor the nonsmooth part $h$ are convex.
The aim is to derive a convergence theory that is in accordance with~\cite{HuangWei:2021:1} also for our intrinsic variant, that is, it is also based on intrinsic properties that $g$ and $h$ have as functions defined on the manifold.
A major difference to~\cite{HuangWei:2021:1} is that the method works intrinsically, \ie entirely on the manifold, without solving a subproblem on a tangent space or requiring an embedding.
This especially allows using Riemannian proximal maps defined intrinsically on the manifold, such as the one introduced in, \eg,~\cite{FerreiraOliveira:2002:1}, instead of Euclidean-flavored approximations thereof.
Furthermore, this approach can directly be applied also to manifolds that are \emph{not} submanifolds of a Euclidean space.
We show two main results:
\begin{itemize}
    \item In \cref{thm:complexity}, we show sublinear convergence to $\varepsilon$-stationary points, with rate $\cO(1/\sqrt k)$.
    \item In \cref{thm:global-convergence}, we show that any accumulation point $p^\ast$ of the sequence of iterates is a stationary point of $f$, \ie, $0_{p^\ast} \in \partial f(p^\ast)$.
\end{itemize}
Crucially, our analysis does not make any assumption on the convexity of either of the summands $g, h$ that the objective is split into, whereas a more standard approach usually involves assuming the nonsmooth part $h$ to be convex.
As a result, our framework captures manifolds of any bounded curvature, be it positive or negative.
Finally, we validate the performance of our method on realistic numerical examples.
We show that in many cases, the proximal map of $h$ either exists in closed-form or is efficiently computable intrinsically on the manifold.

\subsection{Organization}
The remainder of this paper is organized as follows.
After introducing the necessary notation and some preliminary properties in~\cref{section:preliminaries}, we derive the new formulation of the nonconvex proximal gradient method in \cref{section:NCRPG}.
In \cref{section:Convergence}, we discuss convergence properties of our proposed algorithm.
We introduce a generalization of the algorithm that uses retractions and employs a retraction based distance and proximal map in~\cref{section:Retractions}.
\Cref{section:numerics} illustrates the performance of the algorithm and its convergence, and \cref{section:conclusion} concludes the paper.

\section{Preliminaries}
\label{section:preliminaries}

In this section, we recall some notions from Riemannian geometry.
For more details, the reader may wish to consult \cite{DoCarmo:1992:1,Boumal:2023:1}.
Let $\cM$ be a smooth manifold with Riemannian metric $\riemannian{\cdot}{\cdot}$ and let $\covariantDerivativeSymbol$ be its Levi-Civita connection.
Given $p, q \in \cM$, we denote by $C_{pq}$ the set of all piecewise smooth curves $\gamma \colon[0,1] \to \cM$ from~$p = \gamma(0)$ to~$q = \gamma(1)$.
The \emph{arc length} of $\gamma \in C_{pq}$, $L \colon C_{pq} \to \bbR_{\ge 0}$ is defined as
${
	L(\gamma)
	=
	\int_0^1 \riemanniannorm{\dot \gamma}[\gamma(t)] \, \d t
	,
}$
where $\riemanniannorm{\cdot}[p]$ is the norm induced by the Riemannian metric in the tangent space $\tangentSpace{p}$ at the point~$p$.
The inner product given by the Riemannian metric at~$p$ is ${\riemannian{\cdot}{\cdot}[p] \colon \tangentSpace{p} \times \tangentSpace{p} \to \bbR}$.
The explicit reference to the base point will be omitted whenever the base point is clear from context.
The sectional curvature at $p \in \cM$ is denoted by $K_p(X_p,Y_p)$ for tangent vectors $X_p, Y_p \in \tangentSpace{p}$.

The Riemannian distance between two points, denoted with $\dist \colon \cM \times \cM \to \bbR_{\ge 0} \cup \{+\infty\}$, is given by
$
	\dist(p,q)
	=
	\inf_{\gamma \in C_{pq}} L(\gamma)
	.
$
A \emph{geodesic arc} from $p \in \cM$ to $q \in \cM$ is a smooth curve $\gamma \in C_{pq}$ that is parallel along itself, \ie, $\covariantDerivative{\dot \gamma(t)}[\dot \gamma(t)] = 0$ for all $t \in (0,1)$.
A finite-length geodesic arc is \emph{minimal} if its arc length coincides with the Riemannian distance between its extreme points.
Given two points $p, q \in \cM$, we denote with $\geodesic{p}{q}$ a minimal geodesic arc that connects $p = \geodesic{p}{q}(0)$ to $q = \geodesic{p}{q}(1)$.
Given a point $p\in \cM$ and a tangent vector $X_p \in \tangentSpace{p}$, a geodesic can also be specified
as the curve $\geodesic{p}{X_p}$ such that $\geodesic{p}{X_p}(0)=p$ and $\geodesic{p}{X_p}'(0) = X_p$.
The \emph{exponential map} at $p \in \cM$ maps a tangent vector $X_p \in \tangentSpace{p}$ to the endpoint of the corresponding geodesic, $\exponential{p}(X_p) = \geodesic{p}{X_p}(1)$.
If it exists, the inverse map is called the \emph{logarithmic map} $\logarithm{p}\colon \cM \rightarrow \tangentSpace{p}$.
We call a set \emph{uniquely geodesic} if every two points in that space are connected by one and only one geodesic in that set.
In this case the logarithmic map exists everywhere.
A set $\cU \subseteq \cM$ is said to be \emph{(uniquely) geodesically convex} if for any $p,q \in \cU$, there exists a (unique) minimal geodesic arc $\geodesic{p}{q}$ which lies entirely in~$\cU$.
Observe that if a manifold $\cM$ has positive sectional curvature, the set $\cU$ may not be allowed to have arbitrarily large diameter.
For example, if $\kmin > 0$, for $\cU$ to be uniquely geodesically convex, its diameter must be bounded by $\diam(\cU) < \pi/\sqrt{\kmin}$.
See, \eg, \cite[Proposition~4.2~(iii)]{LiYao:2012:1}, \cite[Proposition~4.1]{WangLiYao:2016:1} or \cite[Proposition~II.1.4]{BridsonHaefliger:1999:1}.

We denote the \emph{parallel transport} from~$p$ to~$q$ along~$\geodesic{p}{q}$ with respect to the connection~$\covariantDerivativeSymbol$ with $\parallelTransport{p}{q} \colon \tangentSpace{p} \to \tangentSpace{q}$, defined by
$
	\parallelTransport{p}{q} Y_p
	=
	X_q
$
for
$
	Y_p \in \tangentSpace{p}
	,
$
where $X$ is the unique smooth vector field along $\geodesic{p}{q}$ which satisfies
$
	X_p
	=
	Y_p
	\text{ and }
	\covariantDerivative{\dot \gamma}{X}
	=
	0
	.
$
Notice that $\parallelTransport{p}{q}$ depends on the choice of $\geodesic{p}{q}$.

Given a differentiable function $f \colon \cM \to \bbR$, we denote the \emph{Riemannian gradient} of $f$ at $p \in \cM$ by $\grad f(p) \in \tangentSpace{p}$.
See \cite[Definition~8.57]{Boumal:2023:1} for more details.
Given a uniquely geodesic set $\cU \subseteq \cM$ and a real number $L_f > 0$, the function $f$ is said to be \emph{$L_f$-smooth} over $\cU$ if, for any $p, q \in \cU$,
\begin{equation*}
    \riemanniannorm[auto]
    {\grad f(p) - \parallelTransport{q}{p} \grad f(q)}
    \le
    L_f
    \dist(p, q)
    ;
\end{equation*}
see, \eg, \cite[Definition~4]{ZhangSra:2016:1}, \cf~\cite[Definition~10.44]{Boumal:2023:1}.
Furthermore, it can be shown that, if $f$ is $L_f$-smooth over $\cU$, then
\begin{equation*}
    f(q)
    \le
    f(p)
    +
    \inner{\grad f(p)}{\logarithm{p}(q)}
    +
    \frac{L_f}{2}
    \dist^2(p, q)
    ,
\end{equation*}
for any $p, q \in \cU$; \cf~\cite[Proposition~10.53]{Boumal:2023:1}.
Given a real number $L \ge 0$, a function $f$ is $L$-Lipschitz if, for any $p, q \in \cM$,
\begin{equation*}
    \abs{f(p) - f(q)}
    \le
    L
    \dist(p, q)
    .
\end{equation*}
Furthermore, $f$ is \emph{locally} Lipschitz if, for every $p \in \cM$, there exists a neighborhood $\cU$ of $p$ such that $\restr{f}{\cU}$ is Lipschitz.
Let $f \colon \cM \to \bbR$ be a locally Lipschitz function, and let $\hat f \coloneq f \circ \expOp \colon \tangentBundle \to \bbR$.
Then, the generalized Clarke directional derivative of $f$ in the direction $X_p \in \tangentSpace{p}$ at $p \in \cM$ is defined as
\begin{equation}
    \label{eq:generalized-clarke-directional-derivative}
    f^\circ(p, X_p)
    \coloneq
    \hat f^\circ(0_p, X_p)
    \coloneq
    \limsup_{\substack{Y_p \to 0_p \\ t \downarrow 0}}
    \frac{
        f \circ \exponential{p}(Y_p + tX_p)
        -
        f \circ \exponential{p}(0_p)
    }{t}
    .
\end{equation}
The generalized Clarke subdifferential of $f$ at $p \in \cM$ is then defined as
\begin{equation}
    \label{eq:generalized-clarke-subdifferential}
    \partial f(p)
    \coloneq
    \setDef[auto]{X_p \in \tangentSpace{p}}{
        f^\circ(p, Y_p)
        \ge
        \riemannian{X_p}{Y_p}
        \
        \text{for all }
        Y_p \in \tangentSpace{p}
    }
    .
\end{equation}
Let $f$ be a locally Lipschitz function on $\cM$ and $p^\ast \in \cM$ a local minimizer of $f$, then
\begin{equation*}
    0_{p^\ast}
    \in
    \partial
    f(p^\ast)
    ,
\end{equation*}
see, \eg, \cite[Proposition~3.3]{YangZhangSong:2014:1}.
If now $f = g + h$ with $g$ of class $\cC^1$ and $h$ locally Lipschitz, then
\begin{multline*}
    f^\circ(p, X_p)
    =
    \limsup_{\substack{Y_p \to 0_p \\ t \downarrow 0}}
    \frac{1}{t} \bigl(
            g \circ \exponential{p}(Y_p + tX_p)
            -
            g \circ \exponential{p}(0_p)
        \\ % uncomment for amspreprint
        +
            h \circ \exponential{p}(Y_p + tX_p)
            -
            h \circ \exponential{p}(0_p)
    \bigr)
    ,
\end{multline*}
and since $\limsup(a_n + b_n) \le \limsup a_n + \limsup b_n$, we get
\begin{equation}
    \label{eq:directional-derivative-of-sum-inequality-1}
    f^\circ(p, X_p)
    =
    (g + h)^\circ(p, X_p)
    \le
    g'(p, X_p)
    +
    h^\circ(p, X_p)
    ,
\end{equation}
where $g'(p, X_p) = \riemannian{\grad g(p)}{X_p}$.
Since $h^\circ(p, X_p) = ((h + g) - g)^\circ(p, X_p)$, applying \cref{eq:directional-derivative-of-sum-inequality-1}, we get
\begin{equation}
    \label{eq:directional-derivative-of-sum-inequality-2}
    h^\circ(p, X_p)
    +
    g'(p, X_p)
    \le
    (h + g)^\circ(p, X_p)
    .
\end{equation}
\Cref{eq:directional-derivative-of-sum-inequality-1} and \cref{eq:directional-derivative-of-sum-inequality-2} now imply
\begin{equation}
    \label{eq:directional-derivative-of-sum-equality}
    (g + h)^\circ(p, X_p)
    =
    g'(p, X_p)
    +
    h^\circ(p, X_p)
    .
\end{equation}
This and the definition of the Clarke subdifferential now give
\begin{equation}
    \label{eq:clarke-subdifferential-of-sum}
    \partial (g + h) (p)
    =
    \grad g(p)
    +
    \partial h(p)
    \quad
    \text{for all }
    p \in \cM
    .
\end{equation}
This is a non-Hadamard version of \cite[Lemma~3.1]{BentoFerreiraOliveira:2015:1}.
Therefore, $0 \in \partial(g+h)(p) = \grad g(p) + \partial h(p)$ is equivalent to the stationarity condition
\begin{equation}
    \label{eq:stationarity-condition}
    - \grad g(p)
    \in
    \partial h(p)
    .
\end{equation}

Given a function $f$ and a real number $\lambda > 0$, we introduce the proximal map as
\begin{equation*}
    \prox{\lambda f}(p)
    \coloneq
    \argmin_{q \in \cM}
        f(q)
        +
        \frac{1}{2 \lambda} \dist^2(p, q)
    ,
\end{equation*}
if the minimizer exists.
Note that, as opposed to other works, the minimum is taken \emph{on the manifold}.

In the analysis of the proximal gradient method, we will also need the following curvature-dependent quantities.
Given real numbers $\kappa_1, \kappa_2$, which will typically be lower and upper bounds to the sectional curvature of $\cM$, and $s \in \bbR$, we define the following quantities, see also \cite[Definition~3.8]{LezcanoCasado:2020:1}:
\begin{align}
	\label{eq:hessian-eigs1}
    \newtarget{def:zeta-first}{
        \zeta_{1, \kappa_1}(s)
    }
    &
    \coloneq
    \begin{cases}
        1
        &
        \text{if }
        \kappa_1 \ge 0
        ,
        \\
        \sqrt{- \kappa_1} \, s \coth(\sqrt{- \kappa_1} \, s)
        \quad
        &
        \text{if }
        \kappa_1 < 0
        ,
    \end{cases}
    \\
    \label{eq:hessian-eigs2}
    \newtarget{def:zeta-second}{
        \zeta_{2, \kappa_2}(s)
    }
    &
    \coloneq
    \begin{cases}
        1
        &
        \text{if }
        \kappa_2 \le 0
        ,
        \\
        \sqrt{\kappa_2} \, s \cot(\sqrt{\kappa_2} \, s)
        \quad
        &
        \text{if }
        \kappa_2 > 0
        ,
    \end{cases}\\
	\label{eq:sigma-definition}
    \newtarget{def:sigma-curvature}{
        \sigma_{\kappa_1, \kappa_2}(s)
    }
	&\coloneq
	\max \paren[big]\{\}{%
        \zeta_{1, \kappa_1}(s)
		,
		\;
        \abs{\zeta_{2, \kappa_2}(s)}
	}
	,\\
    \label{eq:generalized-sine-positive-zero}
    \newtarget{def:generalized-sine-positive-zeros}{
        \pikappa
    }
    &\coloneq
    \begin{cases}
        \infty
        &
        \text{if }
        \kappa \le 0
        ,
        \\
        \frac{\pi}{\sqrt{\kappa}}
        &
        \text{if }
        \kappa > 0
        .
    \end{cases}
\end{align}

\section{The Riemannian Proximal Gradient Method}
\label{section:NCRPG}

To derive
the proximal gradient method, the function $g$ is replaced by a proximal regularization of its linearization; see~\cite[Section~2.2]{BeckTeboulle:2009:1} for the derivation in the Euclidean case.
For $\lambda > 0$ and a point $p$ we compute a new candidate as
\begin{equation*}
    \argmin_{q\in\cM}
    g(p)
    +
    \riemannian{\grad g(p)}{\logarithm{p}q}
    +
    \frac{1}{2\lambda} \dist^2(p,q)
    +
    h(q)
    .
\end{equation*}
We introduce $r(p) \coloneq \frac{1}{2\lambda} \dist^2(p,q)$ and Taylor expand this function (\cf~\cite[Section~4.8]{Boumal:2023:1}):
\begin{equation*}
    \begin{split}
    r(\exponential{p}(X_p))
    &
    =
    r(p)
    +
    \riemannian{\grad r(p)}{X_p}
    +
    \cO(\riemanniannorm{X_p}^2)
    ,
    \\
    &
    =
    r(p)
    -
    \frac{1}{\lambda}
    \riemannian{\logarithm{p}q}{X_p}
    +
    \cO(\riemanniannorm{X_p}^2
    )
    ,
    \end{split}
\end{equation*}
where $X_p \in \tangentSpace{p}$.
For the case $X_p = -\lambda\grad g(p)$ this reads
\begin{equation*}
    \frac{1}{2\lambda}
    \dist^2(
        \exponential{p}(-\lambda\grad g(p))
        ,
        q
    )
    =
    \frac{1}{2\lambda}
    \dist^2(p, q)
    +
    \riemannian{\logarithm{p}q}{\grad g(p)}
    +
    \cO(\lambda^2)
    ,
\end{equation*}
using this as a second quadratic approximation we obtain the new candidate as
\begin{equation*}
    \argmin_{q\in\cM}
    g(p)
    +
    \frac{1}{2\lambda}
    \dist^2(
        \exponential{p}(-\lambda\grad g(p))
        ,
        q
    )
    +
    h(q)
    .
\end{equation*}
Since $g(p)$ does not change the minimizer, the resulting expression is equal to the proximal map
\begin{equation*}
    \prox[big]{\lambda h}(
        \exponential{p}(-\lambda\grad g(p))
    )
    =
    \argmin_{q\in\cM}
    \frac{1}{2\lambda}
    \dist^2(
        \exponential{p}(-\lambda\grad g(p))
        ,
        q
    )
    +
    h(q)
    .
\end{equation*}

In this article, we consider the general case where the function $h$ or the Riemannian distance are not necessarily geodesically convex.
Then, the proximal map $\prox{\lambda h}$ might not be single-valued, and the iteration might not be well-defined.
In this case, inspired by \cite[Equation~(3.1)]{HuangWei:2021:1}, we only demand finding a stationary point $p^\ast$ of the function
\begin{equation}
    \label{eq:nonconvex-subproblem-objective}
    H(q; p, \lambda)
    \coloneq
    h(q)
    +
    \frac{1}{2 \lambda}
    \dist^2
    \left(
        q
        ,
        \exponential[big]{p}(
            -\lambda \grad g(p))
    \right)
    ,
\end{equation}
\suto $H(p^\ast; p, \lambda) \le H(p; p, \lambda)$.
We state the nonconvex Riemannian Proximal Gradient method (NCRPG) in \Cref{algorithm:NCRPG}.
\begin{algorithm}[htp]
	\caption{Nonconvex Riemannian Proximal Gradient Method (NCRPG)}
	\label{algorithm:NCRPG}
	\begin{algorithmic}[1]
		\Require
        $g$, %
        $\grad g$, %
        $h$, %
        a sequence $\sequence{\lambda}{k}$, %
		an initial point $\sequence{p}{0} \in \cM$.
		\While{convergence criterion is not fulfilled}
    \State Find a stationary point $\sequence{p}{k+1}$ of $H(\cdot; \sequence{p}{k}, \sequence{\lambda}{k})$
    \Statex \hspace{0.45cm} \suto~$H(\sequence{p}{k+1}; \sequence{p}{k}, \sequence{\lambda}{k})
    		\le H(\sequence{p}{k};
		    \sequence{p}{k},
		    \sequence{\lambda}{k}
		)$.

		\State Set $k \coloneq k+1$
		\EndWhile
	\end{algorithmic}
\end{algorithm}
We point out that, in our experiments, the proximal subproblem is either solvable in closed-form or can be efficiently solved iteratively on the manifold.
See \cref{section:numerics} for more details.

\section{Convergence}
\label{section:Convergence}

In this section, we analyze the convergence behavior of the nonconvex Riemannian proximal gradient (NCRPG) method from~\cref{algorithm:NCRPG}.
We will assume the curvature of the manifold to be bounded, and no convexity assumptions are made for the functions $g$ and $h$.

The following assumptions are made in this section:
\begin{assumption}\leavevmode\par
	\label{assumption:nonconvex-assumptions}
	\begin{enumerate}

        \item \label{assumption:nonconvex-assumptions:boundedsectionalcurvature}
            There exist two numbers $\kmin, \kmax \in \bbR$ such that the sectional curvature satisfies $\newtarget{def:k-min}{\kmin} \le K_p(X_p,Y_p) \le \newtarget{def:k-max}{\kmax}$ for all $p \in \cM$ and all linearly independent unit vectors $X_p , Y_p\in \tangentSpace{p}$.

        \item \label{assumption:nonconvex-assumptions:compact_levelset}
            For every starting point $\sequence{p}{0} \in \cM$ the sublevel set
            \begin{equation*}
                \newtarget{def:level-set}{
                    \startlevelset
                }
                \coloneq
                \setDef[auto]{p \in \cM}{
                    f(p)
                    \le
                    f(\sequence{p}{0})
                }
            \end{equation*}
            is compact in $\cM$.

        \item \label{assumption:nonconvex-assumptions:g}
            $g \colon \cM \to \bbR$ is continuously differentiable and $\newtarget{def:gradg-lipschitz-constant}{\lipgrad}$-smooth on $\startlevelset$.

		\item \label{assumption:nonconvex-assumptions:h}
            $h \colon \cM \to \bbR$ is locally Lipschitz continuous and lower bounded.

        \item \label{assumption:nonconvex-assumptions:decreaseprox}
            In each step we find a stationary point $\sequence{p}{k+1} \in \cM$ of
            \begin{equation*}
                H(q;
                \sequence{p}{k},
                \sequence{\lambda}{k}
                )
                \coloneq
                h(q)
                +
                \frac{1}{2 \sequence{\lambda}{k}}
                \dist^2
                \Bigl(q,
                \exponential[big]{\sequence{p}{k}}(
                    -\sequence{\lambda}{k} \grad g(\sequence{p}{k})
                )
                \Bigr)
                ,
            \end{equation*}
            \suto $H(\sequence{p}{k+1};
                \sequence{p}{k},
                \sequence{\lambda}{k}
            )
            \le
            H(\sequence{p}{k};
            \sequence{p}{k},
            \sequence{\lambda}{k}
            )$.
    \end{enumerate}
\end{assumption}

\begin{remark}
    \label{remark:nonconvex-assumptions:boundedness-of-h-and-grad-g}
    Note that since our sublevel set is compact and our objective function $f$ is continuous, it is lower bounded by $f_{opt} \coloneq \min_{q \in \startlevelset} f(q)$.
    Furthermore, since $h$ and $\grad g$ are continuous we obtain bounds $\hlowerbound, \hupperbound, \gradgupperbound \in \bbR$ such that
    \begin{equation*}
        \newtarget{def:h-lower-bound}{
            \hlowerbound
        }
        \le
        h(q)
        \le
        \newtarget{def:h-upper-bound}{
            \hupperbound
        }
        \text{ for all } q \in \startlevelset,
        \text{ and }
        \norm{\grad g(q)}
        \le
        \newtarget{def:gradg-upper-bound}{
            \gradgupperbound
        }
        \text{ for all } q \in \startlevelset.
    \end{equation*}
\end{remark}

Let $p \in \cM$ and $\lambda > 0$.
We introduce the iteration map $\imapOp[auto]{g}{h}{\lambda} \colon \cM \to \cM$ that to each point $p \in \cM$ assigns a stationary point $\newtarget{def:iteration-map}{\imap[auto]{g}{h}{\lambda}{p}} \in \cM$ of the function $H(\cdot ;p, \lambda)$ defined in \cref{eq:nonconvex-subproblem-objective}.
Note that, if $h$ is geodesically convex, then
\begin{equation}
    \label{eq:iteration-mapping}
    \imap[auto]{g}{h}{\lambda}{p}
    =
    \prox{\lambda h}\Bigl(
        \exponential[big]{p}(-\lambda \grad g(p))
    \Bigr)
    \in
    \cM
    .
\end{equation}
Furthermore, we introduce the gradient map $\gmapOp[auto]{g}{h}{\lambda} \colon \cM \to \tangentBundle$ defined by
\begin{equation}
    \label{eq:gradient-mapping}
    \newtarget{def:gradient-map}{
        \gmap[auto]{g}{h}{\lambda}{p}
    }
    \coloneq
    - \frac{1}{\lambda}
    \logarithm[big]{p}(
        \imap[auto]{g}{h}{\lambda}{p}
    )
    \in
    \tangentSpace{p}
    .
\end{equation}
When $g, h$ are clear from context, we denote $\newtarget{def:iteration-map}{\Imap[auto]{\lambda}{p}} \equiv \imap[auto]{g}{h}{\lambda}{p}$, and $\newtarget{def:gradient-map}{\Gmap{\lambda}{p}} \equiv \gmap[auto]{g}{h}{\lambda}{p}$, respectively.

\subsection{Sufficient decrease}
This first step towards a convergence result is to show a sufficient decrease of the cost function. For this, we need the following technical lemma, that provides a curvature-dependent upper bound on the distance between the points $p$ and $\Imap{\lambda}{p}$ and ensures the strict positivity of $\zeta_{2, \kmax}$ within the sublevel set $\startlevelset$.
Notably, in case $\kmax >0$,  this lemma is crucial as it ensures that the gradient step ${\gradientstep{p} = \exponential{p}(-\lambda \grad g(p))}$ remains within a bounded distance from $p$ and from $\Imap{\lambda}{p}$ that is small enough to ensure that the logarithmic map, and hence also the gradient mapping $\Gmap{\lambda}{p}$, is defined whenever it is needed.
\begin{lemma}
    \label{lemma:lambda-bound}
    Let $p \in \startlevelset$, $\delta > 0$, $\hlowerbound$, and $\hupperbound$ be the numbers in \cref{remark:nonconvex-assumptions:boundedness-of-h-and-grad-g},
    \begin{equation}
        \label{eq:lambda-epsilon}
        \newtarget{def:delta-stepsize}{
            \lambdadelta
        }
        \coloneq
        \frac{
                \sqrt{
                    4 \, (\hupperbound - \hlowerbound)^2
                    +
                    \frac{{\pikmax}^2}{(2 + \delta)^2} \, \gradgupperbound^2
                }
                -
                2 \, (\hupperbound - \hlowerbound)
        }{2 \, \gradgupperbound^2}
        ,
    \end{equation}
    and
    \begin{equation}
        \label{eq:zeta-delta}
        \newtarget{def:delta-curvature}{
            \zetadelta
        }
        \coloneq
        \zeta_{2, \kmax}\left(
            \deltaradius
        \right)
        =
        \begin{cases}
			1
			&
			\text{if }
			\kmax \le 0
			,
			\\
      \frac{\pi}{2 + \delta} \cot \left(\frac{\pi}{2 + \delta} \right)
			\quad
			&
			\text{if }
			\kmax > 0
			,
		\end{cases}
    \end{equation}
    with
    $\newtarget{def:delta-radius}{\deltaradius} \coloneq \frac{\pikmax}{2 + \delta}$.
    Then, for all $\lambda \in (0, \lambdadelta]$ we have $\Imap{\lambda}{p}, \gradientstep{p} \in \cB(p, \deltaradius)$, where $\cB(p, \deltaradius)$ is the uniquely geodesically convex open ball centered at $p$ with radius $\deltaradius$ and  
    $\newtarget{def:gradient-step}{\gradientstep{p} = \exponential{p}(-\lambda \grad g(p))}$.
    Furthermore, for all $q \in \cB(p, \deltaradius)$, we have
    \begin{equation}
        \label{eq:zetasecond-dist-bound}
        \zetasecond{\dist(p, q)} 
        \ge
        \zetadelta
        >
        0
        .
    \end{equation}
\end{lemma}
\begin{proof}
    Let $\delta > 0$, and observe that 
    $\Imap{\lambda}{p}, \gradientstep{p} \in \cB(p, \deltaradius)$
    is achieved, \eg, if both $\dist(p, \gradientstep{p})$ and $\dist(\Imap{\lambda}{p}, \gradientstep{p})$ are upper bounded by $\frac \deltaradius 2$.
    We have $\dist(p, \gradientstep{p}) = \lambda \riemanniannorm{\grad g(p)}$ and
    \cref{assumption:nonconvex-assumptions:decreaseprox} in \cref{assumption:nonconvex-assumptions} implies
    \begin{align*}
        h(\Imap{\lambda}{p})
        +
        \frac{1}{2 \lambda} \dist^2(\Imap{\lambda}{p}, \gradientstep{p})
        \le
        h(p)
        +
        \frac{1}{2 \lambda} \dist^2(p, \gradientstep{p})
        =
        h(p)
        +
        \frac{1}{2 \lambda} \lambda^2 \riemanniannorm{\grad g(p)}^2
        ,
    \end{align*}
    so that
    \begin{equation}
        \dist^2(\Imap{\lambda}{p}, \gradientstep{p})
        \le
        2 \lambda \left(
            h(p)
            -
            h(\Imap{\lambda}{p})
        \right)
        +
        \lambda^2 \riemanniannorm{\grad g(p)}^2
        .
    \end{equation}
    By \cref{assumption:nonconvex-assumptions:h}  and \cref{remark:nonconvex-assumptions:boundedness-of-h-and-grad-g}, we further obtain
    \begin{equation}
        \dist^2(\Imap{\lambda}{p}, \gradientstep{p})
        \le
        2 \lambda \left(
            \hupperbound
            -
            \hlowerbound
        \right)
        +
        \lambda^2 \gradgupperbound^2
        .
    \end{equation}
    Imposing now that
    \begin{equation*}
        2 \lambda \left(
            \hupperbound
            -
            \hlowerbound
        \right)
        +
        \lambda^2 \gradgupperbound^2
        \le
        \left(
            \frac{\deltaradius}{2}
        \right)^2
        ,
    \end{equation*}
    we obtain $\lambda \in \left(0, \lambdadelta \right]$, with $\lambdadelta$ is as in \cref{eq:lambda-epsilon}.
    Also note that this implies 
    $\lambda \le \frac {\deltaradius}{2 \gradgupperbound}$, 
    which in turn implies $\dist(p, \gradientstep{p}) \le \frac \deltaradius 2$.
    Putting everything together we obtain 
    $\Imap{\lambda}{p}, \gradientstep{p} \in \cB(p, \deltaradius)$. The monotonicity of $\zetasecond{s}$ on $(0, \deltaradius]$ then implies \cref{eq:zetasecond-dist-bound}, 
    with $\zetadelta$ as in \cref{eq:zeta-delta}.
\end{proof}

With the above Lemma, we are able to show sufficient decrease of the cost function.
\begin{lemma}
    \label{lemma:sufficient-decrease}
    Let $g$ and $h$ satisfy \cref{assumption:nonconvex-assumptions} and $\delta > 0$.
    Then, for any $p \in \startlevelset$ and ${\lambda \in \left(0, \min\bigl\{\lambdadelta, \frac{\zetadelta}{\lipgrad}\bigr\}\right)}$ the following inequality holds
    \begin{equation}
        \label{eq:sufficient-decrease}
        f(p)
        -
        f(\Imap[auto]{\lambda}{p})
        \ge
        \frac{
            \lambda\zetadelta
            -
            \lambda^2\lipgrad
        }{2}
        \riemanniannorm{\Gmap{\lambda}{p}}^2
        .
    \end{equation}
\end{lemma}
\begin{proof}
    Let $p \in \startlevelset$.
    Since $\Imap[auto]{\lambda}{p}$ is a stationary point of the function $H(\cdot ;p, \lambda)$ defined in \cref{assumption:nonconvex-assumptions},~\cref{assumption:nonconvex-assumptions:decreaseprox} and $H(\Imap[auto]{\lambda}{p}; p, \lambda) \le H(p; p,\lambda)$, we have
    \begin{align*}
        f(p)
        &
        \ge
        g(\Imap[auto]{\lambda}{p})
        -
        \riemannian{\grad g(p)}{\logarithm{p}\Imap[auto]{\lambda}{p}}
        -
        \frac{\lipgrad}{2} \dist^2 \left(
            p
            ,
            \Imap[auto]{\lambda}{p}
        \right)
        +
        h(\Imap[auto]{\lambda}{p})
        \\
        &
        \quad
        +
        \frac{1}{2 \lambda} \dist^2 \left(
            \exponential{p}(-\lambda \grad g(p))
            ,
            \Imap[auto]{\lambda}{p}
        \right)
        -
        \frac{1}{2 \lambda} \dist^2 \left(
            p
            ,
            \exponential{p}(-\lambda \grad g(p))
        \right)
        ,
    \end{align*}
    where we also used \cref{assumption:nonconvex-assumptions:g}.
    Let now ${\gradientstep{p} \coloneq \exponential{p}(-\lambda \grad g(p))}$.
    We get
    \begin{multline*}
      % \hspace{-1.15em} % comment out for amspreprint
        f(p)
        \ge
        f(\Imap{\lambda}{p})
        +
        \frac{1}{2 \lambda}
        \left[
            \dist^2(\Imap{\lambda}{p}, \gradientstep{p})
            -
            \dist^2(p, \gradientstep{p})
            -
            \lipgrad \lambda \dist^2(p, \Imap{\lambda}{p})
        \right.
        \\ % uncomment for amspreprint
        \left.
            +
            2 \riemannian{\logarithm{p}\gradientstep{p}}{\logarithm{p}\Imap{\lambda}{p}}
        \right]%
        .
    \end{multline*}
    We are now in a position to apply the law of cosines on manifolds; see, \eg, \cite[Theorem~3.15]{LezcanoCasado:2020:1}, or~\cite[Corollary~2.1]{AlimisisBecigneulLucchiOrvieto:2020:1} which in this case reads
    \begin{equation*}
        \dist^2(\Imap{\lambda}{p}, \gradientstep{p})
        \ge
        \zetasecond{\dist(m, p)}
        \dist^2(p, \Imap{\lambda}{p})
        +
        \dist^2(p, \gradientstep{p})
        -
        2 \riemannian{\logarithm{p}\gradientstep{p}}{\logarithm{p}\Imap{\lambda}{p}}
        ,
    \end{equation*}
    where $m$ is a point on the minimal geodesic arc connecting $\Imap{\lambda}{p}$ and $\gradientstep{p}$.
    Since $\lambda \leq \lambdadelta$, we obtain from \cref{lemma:lambda-bound} that 
    $\Imap{\lambda}{p}, \gradientstep{p} \in \cB(p, \deltaradius)$ and because this ball is uniquely geodesically convex, 
    $m \in \cB(p, \deltaradius)$, and thus  $\zetasecond{\dist(m, p)} \ge \zetadelta >0$.
    This yields
    \begin{align*}
        f(p)
        &
        \ge
        f(\Imap{\lambda}{p})
        +
        \frac{
            \frac{\zetadelta}{\lambda}
            -
            \lipgrad
        }{2}
        \dist^2(p, \Imap{\lambda}{p})
        .
    \end{align*}
    Choosing a stepsize $\lambda \in \left(0, \min\left\{\lambdadelta, \frac{\zetadelta}{\lipgrad}\right\}\right)$ gives the claim.
\end{proof}

\subsection{Stepsizes}\label{subsection:stepsize-discussion}

In \cref{lemma:sufficient-decrease}, it was necessary to choose a stepsize $\lambda$.
We now present several strategies for this choice.
Notice that, whenever $\cM$ has nonpositive curvature $\kmax \le 0$, it holds that $\zetadelta = 1$.
This means that to guarantee sufficient decrease, the stepsize $\lambda$ can be then chosen as $\lambda \in \left(0, \frac{1}{\lipgrad} \right)$ in this case.

We will consider a constant stepsize and a backtracking procedure.
For the constant stepsize strategy, we choose $\sequence{\lambda}{k} \coloneq \bar \lambda \in \left(0, \min\left\{\lambdadelta, \frac{\zetadelta}{\lipgrad}\right\}\right)$.
The backtracking procedure is described in \Cref{algorithm:backtracking}.
\begin{algorithm}[htp]
	\caption{Backtracking procedure for stepsize selection}
	\label{algorithm:backtracking}
	\begin{algorithmic}[1]
		\Require
        $f$,
        current $\sequence{p}{k}$,
        initial guess $s \in (0, \lambdadelta)$,
        constants $\eta \in (0,1)$ and $\beta \in (0,\frac \zetadelta 2)$.
        \State Set $\sequence{\lambda}{k} = s$
		\While{$
            \label{eq:backtracking-sufficient-decrease}
            f(\sequence{p}{k})
            -
            f\left(
                \Imap[auto]{\sequence{\lambda}{k}}{\sequence{p}{k}}
            \right)
            <
             \beta
            \sequence{\lambda}{k}
            \riemanniannorm[auto]{
                \Gmap{\sequence{\lambda}{k}}{\sequence{p}{k}}
            }^2$}
		\State Set $\sequence{\lambda}{k} = \eta \sequence{\lambda}{k}$
		\EndWhile
	\end{algorithmic}
\end{algorithm}
This means we look for the smallest $i_k \in \bbN$ such that
\begin{equation}\label{eq:armijo-condition}
    f(\sequence{p}{k})
    -
    f(
        \Imap{s \eta^{i_k}}{\sequence{p}{k}}
    )
    \ge
     \beta
    s \eta^{i_k}
    \riemanniannorm[auto]{
        \Gmap{s\eta^{i_k}}{\sequence{p}{k}}
    }^2
\end{equation}
holds.
Notice that this procedure is finite since for
\begin{equation}
    \label{eq:backtracking-lambda-upper-bound}
    \sequence{\lambda}{k}
    <
    \min \left\{
        \lambdadelta
        ,
        \frac{\zetadelta - 2 \beta}
        {\lipgrad}
    \right\}
    ,
\end{equation}
we have
$
    \frac{\zetadelta- \sequence{\lambda}{k} \lipgrad}
    {2}
    \ge
     \beta
    ,
$
and thus by \cref{lemma:sufficient-decrease}
\begin{align*}
    f(\sequence{p}{k})
    -
    f(
        \Imap{\sequence{\lambda}{k}}{\sequence{p}{k}}
    )
    \ge
    \sequence \lambda k
    \frac{
        \zetadelta
        -
        \sequence{\lambda}{k}\lipgrad}
    {2}
    \riemanniannorm[auto]{\Gmap{\sequence{\lambda}{k}}{\sequence{p}{k}}}^2
    \ge
    \beta
    \sequence \lambda k
    \riemanniannorm[auto]{\Gmap{\sequence{\lambda}{k}}{\sequence{p}{k}}}^2
    ,
\end{align*}
so the procedure must end as soon as $\sequence{\lambda}{k}$ satisfies \cref{eq:backtracking-lambda-upper-bound}.

We can use this observation to obtain a lower bound for $\sequence \lambda k$.
Let $\sequence \lambda k$ be the stepsize obtained by the backtracking procedure.
Then either $\sequence \lambda k = s$, or $\frac{\sequence \lambda k}{\eta}$ has not fulfilled \cref{eq:backtracking-sufficient-decrease}, meaning
\begin{equation}
    \label{eq:backtracking-lambda-lower-bound}
    \sequence{\lambda}{k}
    \ge
    \min \left\{
        s,
        \eta\lambdadelta,
        \eta\frac{\zetadelta - 2 \beta}
        {\lipgrad}
    \right\}
    .
\end{equation}
Taking either of the stepsizes above allows us to obtain a sufficient decrease property.
\begin{corollary}
    \label{lemma:sufficient-decrease-with-constant}
    Suppose that \cref{assumption:nonconvex-assumptions} hold.
    Let $\sequence(){p}{k}$ be the sequence generated by \cref{algorithm:NCRPG} for some starting point $\sequence p 0$ and either a constant stepsize $\sequence{\lambda}{k} \coloneq \bar \lambda \in \left(0, \min\left\{\lambdadelta, \frac{\zetadelta}{\lipgrad}\right\}\right)$ or the backtracking procedure with constants $(s,  \beta, \eta)$ as described in \cref{subsection:stepsize-discussion}.
    Then for any $k \ge 0$,
    \begin{equation}
        \label{eq:sufficient-decrease-with-constant}
        f(\sequence{p}{k})
        -
        f(\sequence{p}{k+1})
        \ge
        M
        \riemanniannorm[auto]{
            \Gmap{\sequence \lambda k}{\sequence p k}
        }^2,
    \end{equation}
    where
    \begin{equation}
        \label{eq:sufficient-decrease-M-d}
        \begin{aligned}
            M
            &
            \coloneq
            \begin{cases}
                \frac{
                    \bar\lambda\zetadelta
                    -
                    \bar\lambda^2\lipgrad
                }{2}
                &
                \text{constant stepsize}
                ,
                \\
                 \beta
                 \min \left\{
                    s,
                    \eta\lambdadelta,
                    \eta\frac{\zetadelta - 2 \beta}
                    {\lipgrad}
                \right\}
                &
                \text{backtracking}
                .
            \end{cases}
        \end{aligned}
    \end{equation}
\end{corollary}

\subsection{Stationary points are zeros of the gradient mapping}
The following result gives a characterization of the stationary points of the optimization problem \cref{eq:splitting} in terms of the gradient mapping.
\begin{theorem}
    \label{thm:zero-g-stationarity}
    Suppose that \cref{assumption:nonconvex-assumptions} hold.
    Then $\gmap[auto]{g}{h}{\lambda}{p^\ast} = 0$ if and only if $p^\ast \in \startlevelset$ is a stationary point of \cref{eq:splitting}.
\end{theorem}
\begin{proof}
    Observe that $\gmap{g}{h}{\lambda}{p^\ast} = 0 \in \tangentSpace{p^\ast}$ if and only if
    \begin{equation*}
        p^\ast
        =
        \imap{g}{h}{\lambda}{p^\ast}
        .
    \end{equation*}
    This is equivalent to
    $
        H(p^\ast;
            p^\ast,
            \lambda)
        =
        H(\imap{g}{h}{\lambda}{p^\ast};
            p^\ast,
            \lambda)
    $
    with
    \begin{equation*}
        0
        \in
        \partial H(p^\ast;
            p^\ast,
            \lambda)
        =
        \partial h(p^\ast)
        -
        \frac {1}{\lambda}
        \logarithm[big]{
            p^\ast
        }
        (
            \exponential{p^\ast}(-\lambda \grad g(p^\ast))
        )
        ,
    \end{equation*}
    which is equivalent to
    \begin{equation*}
        \frac{1}{\lambda} \logarithm[big]{p^\ast}(\exponential{p^\ast}(-\lambda \grad g(p^\ast)))
        \in
        \partial h(p^\ast)
        ,
    \end{equation*}
    \ie, if and only if
    \begin{equation*}
        - \grad g(p^\ast)
        \in
        \partial h(p^\ast)
        ,
    \end{equation*}
    which is the condition for stationarity given in \cref{eq:stationarity-condition}.
\end{proof}

\begin{lemma}
    \label{lemma:nonincrease-stationarity}
    Suppose that \cref{assumption:nonconvex-assumptions} hold.
    Let $\sequence(){p}{k}$ be the sequence generated by \cref{algorithm:NCRPG} for some starting point $\sequence p 0$
    and either a constant stepsize $\sequence{\lambda}{k} \coloneq \bar \lambda \in \left(0, \min\left\{\lambdadelta, \frac{\zetadelta}{\lipgrad}\right\}\right)$ or the backtracking procedure with constants $(s,  \beta, \eta)$.
    Then,
    \begin{enumerate}
        \item
            the sequence $\left(f(\sequence{p}{k})\right)$ is nonincreasing, and $f(\sequence{p}{k+1}) < f(\sequence{p}{k})$ if and only if $\sequence{p}{k}$ is not a stationary point of \cref{eq:splitting}; in particular $\sequence(){p}{k} \subseteq \startlevelset$.
        \item
            $\Gmap{\sequence \lambda k}{\sequence p k} \to 0$ as $k \to \infty$;
        \item
            $\min_{n = 0, \hdots, k} \riemanniannorm{\Gmap{\sequence \lambda n}{\sequence{p}{n}}}
                \le
                \sqrt{\frac
                {f(\sequence p 0)- f_{opt}}
                {M(k+1)}}$,
            where $M$ is given in \cref{eq:sufficient-decrease-M-d} and $f_{opt}$ is the optimal solution of $f$ in the sublevel set $\startlevelset$ as described in \cref{remark:nonconvex-assumptions:boundedness-of-h-and-grad-g}.
    \end{enumerate}
\end{lemma}
\begin{proof}
    \begin{enumerate}
        \item
            If $\sequence{p}{k}$ is not a stationary point of \cref{eq:splitting}, then $\Gmap{\sequence \lambda k}{\sequence{p}{k}} \neq 0$ by \cref{thm:zero-g-stationarity}.
            Thus, by \cref{eq:sufficient-decrease-with-constant} we have $f(\sequence{p}{k+1}) < f(\sequence{p}{k})$.
            Vice versa, if $\sequence{p}{k}$ is a stationary point of \cref{eq:splitting}, then $\Gmap{\sequence \lambda k}{\sequence{p}{k}} = 0$, which implies $\sequence{p}{k+1} = \sequence{p}{k}$, hence $f(\sequence{p}{k}) = f(\sequence{p}{k+1})$.

        \item
            The sequence $\left(f(\sequence{p}{k})\right)$ converges because it is nonincreasing and bounded from below.
            This implies $f(\sequence{p}{k}) - f(\sequence{p}{k+1}) \to 0$ as $k \to \infty$, which combined with \cref{lemma:sufficient-decrease} yields $\Gmap{\sequence \lambda k}{\sequence{p}{k}} \to 0$.

        \item Summing \cref{eq:sufficient-decrease-with-constant} over $n=0, \hdots, k$ yields
            \begin{align*}
                f(\sequence{p}{0})
                -
                f_{opt}
                &
                \ge
                f(\sequence{p}{0})
                -
                f(\sequence{p}{k+1})
                \ge
                M
                \sum_{n=0}^k
                \riemanniannorm[auto]{
                    \Gmap{\sequence \lambda n}{\sequence p n}
                }^2
                \\
                &
                \ge
                M(k+1)
                \min_{n = 0, \hdots, k}
                \riemanniannorm[auto]{
                    \Gmap{\sequence \lambda n}{\sequence{p}{n}}
                }^2
                .
            \qedhere
            \end{align*}
    \end{enumerate}
\end{proof}

\subsection{Sublinear convergence to $\varepsilon$-stationary points}
The next lemma shows that the subdifferential of the cost function at each iteration is bounded by the norm of gradient mapping.
\begin{lemma}
    \label{lemma:bound-on-subdifferential}
    Let $g$ and $h$ satisfy \cref{assumption:nonconvex-assumptions}.
    Let $\sequence(){p}{k}$ be the sequence generated by \cref{algorithm:NCRPG} for some starting point $\sequence p 0$
    and either a constant stepsize $\sequence{\lambda}{k} \coloneq \bar \lambda \in \left(0, \min\left\{\lambdadelta, \frac{\zetadelta}{\lipgrad}\right\}\right)$ or the backtracking procedure with constants $(s,  \beta, \eta)$.
    Then, for every $k \in \bbN$, there exists a tangent vector $X_{\sequence{p}{k}} \in \partial f\left(\sequence{p}{k}\right)$ such that
    \begin{equation*}
        \riemanniannorm{X_{\sequence{p}{k}}}
        \le
        (
            \lipgrad
            +
            \sigmakmax
        )
        \sequence \lambda k
        \,
        \riemanniannorm[auto]{
            \Gmap{\sequence \lambda k}{\sequence{p}{k}}
        }
        ,
    \end{equation*}
    where
    \begin{equation}
        \label{eq:stationary-sigma-definition}
        \newtarget{def:curvature-upper-bound}{
            \sigmakmax
        }
        \coloneq
        \begin{cases}
            \sigmacurvature{
                2 \deltaradius
            }
            &
            \text{if } \kmax > 0
            ,
            \\
            \sigmacurvature{
                t \, \gradgupperbound
                +
                \diam(\startlevelset)
            }
            &
            \text{if } \kmax \le 0
            ,
        \end{cases}
    \end{equation}
    with $\deltaradius$ as defined in \cref{lemma:lambda-bound}.
    Here, $t$ is chosen as $\lambda$ if a constant stepsize is used, and as the initial guess $s$ in the backtracked case.
\end{lemma}
\begin{proof}
    Let $\gradientstep{\sequence{p}{k}} \coloneq \exponential{\sequence{p}{k}}({-\sequence{\lambda}{k} \grad g(\sequence{p}{k})})$, and $\sequence{Y}{k+1} \coloneq \grad g(\sequence{p}{k+1})$, $\sequence{Z}{k+1} \coloneq \logarithm{\sequence{p}{k+1}}(\gradientstep{\sequence{p}{k}})$.
    Since $\sequence{p}{k+1}$ is stationary for
    $
        h(p)
        +
        \frac{1}{2 \sequence{\lambda}{k}}
        \dist^2(
            p
            ,
            \gradientstep{\sequence{p}{k}}
        ),
    $
    we have
    \begin{align*}
        0_{\sequence{p}{k+1}}
        \in
        \partial h(\sequence{p}{k+1})
        -
        \frac {1}{\sequence{\lambda}{k}}
        \sequence{Z}{k+1}
        =
        \partial f(
            \sequence{p}{k+1}
        )
        -
        \sequence{Y}{k+1}
        -
        \frac{1}{\sequence{\lambda}{k}}
        \sequence{Z}{k+1}
        ,
    \end{align*}
    where we used the relation $\partial f(p) = \partial h(p) + \grad g(p)$; see \cref{eq:clarke-subdifferential-of-sum}.
    Rearranging the terms we get
    \begin{equation*}
        \sequence{Y}{k+1}
        +
        \frac{1}{\sequence{\lambda}{k}}
        \sequence{Z}{k+1}
        \in
        \partial f(\sequence{p}{k+1})
        .
    \end{equation*}
    We estimate the norm as follows:
    \begin{equation}
        \label{eq:eps_stat_first_part}
        \begin{aligned}
            &
            \riemanniannorm[auto]{
                \sequence{Y}{k+1}
                +
                \frac{1}{\sequence{\lambda}{k}}
                \sequence{Z}{k+1}
            }
            \\
            &
            \qquad \le
            \riemanniannorm[auto]{
                \sequence{Y}{k+1}
                -
                \parallelTransport{\sequence{p}{k}}{\sequence{p}{k+1}}
                \sequence{Y}{k}
            }
            +
            \riemanniannorm[auto]{
                \parallelTransport{\sequence{p}{k}}{\sequence{p}{k+1}}
                \sequence{Y}{k}
                +
                \frac{1}{\sequence{\lambda}{k}}
                \sequence{Z}{k+1}
            }
            \\
            &
            \qquad
            \le
            \lipgrad \dist(\sequence{p}{k}, \sequence{p}{k+1})
            +
            \frac{1}{\sequence{\lambda}{k}}
            \riemanniannorm[auto]{
                \parallelTransport{\sequence{p}{k}}{\sequence{p}{k+1}}
                \sequence{W}{k}
                -
                \sequence{Z}{k+1}
            }
            ,
        \end{aligned}
    \end{equation}
    where the second inequality is due to the $\lipgrad$-smoothness of the function $g$, and we also used ${\sequence{W}{k} \coloneq \logarithm{\sequence{p}{k}}(\gradientstep{\sequence{p}{k}}) = - \sequence{\lambda}{k}\sequence{Y}{k}}$.
    It holds that
    \begin{equation*}
        \riemanniannorm[auto]{
            \parallelTransport{\sequence{p}{k}}{\sequence{p}{k+1}}
            \sequence{W}{k}
            -
            \sequence{Z}{k+1}
        }
        \le
        \sigmacurvature{\dist(m, \gradientstep{\sequence{p}{k}})}
        \dist(\sequence{p}{k+1}, \sequence{p}{k})
        ,
    \end{equation*}
    where $m$ is a point on the geodesic arc that connects $\sequence{p}{k}$ and $\sequence{p}{k+1}$.
    This is an implication of \cite[Appendix~C]{AlimisisOrvietoBecigneulLucchi:2021:1} with the notation adapted to our setting.
    Using the triangle inequality together with the fact that $m$ lies on such geodesic, we have $\dist(m, \gradientstep{\sequence{p}{k}}) \le \dist(\sequence{p}{k}, \gradientstep{\sequence{p}{k}}) + \dist(\sequence{p}{k}, \sequence{p}{k+1})$.
    Thus by \cref{lemma:lambda-bound}, $\dist(m, \gradientstep{\sequence{p}{k}})$ is bounded by $2 \deltaradius < \pikmax$ for all $k \in \bbN$, which in turn implies $\sigmacurvature{\dist(m, \gradientstep{\sequence{p}{k}})} \le \sigmacurvature{2 \deltaradius}$ for all $k \in \bbN$ if $\kmax > 0$.
    If instead $\kmax \le 0$, $\pikmax$ and $\deltaradius$ are unbounded, and we thus need to argue with the boundedness of the sublevel sets. We get
    \begin{equation*}
        \dist(m , \gradientstep{\sequence{p}{k}})
        \le
        \dist(\sequence{p}{k}, \gradientstep{\sequence{p}{k}})
        +
        \dist(\sequence{p}{k}, \sequence{p}{k+1})
        \le
        t
        \gradgupperbound
        +
        \diam(\startlevelset)
        ,
    \end{equation*}
    where $t$ is chosen as $\lambda$ if a constant stepsize is used, or as the initial guess $s$ in the backtracked case, and $\gradgupperbound$ is the bound on $\grad g$ on $\startlevelset$; see \cref{remark:nonconvex-assumptions:boundedness-of-h-and-grad-g}.
    This implies ${\sigmacurvature{\dist(m, \gradientstep{\sequence{p}{k}})} \le \sigmacurvature{t \gradgupperbound + \diam(\startlevelset)}}$ for all $k \in \bbN$, which is finite by \cref{assumption:nonconvex-assumptions:compact_levelset}.
    Inserting \cref{eq:stationary-sigma-definition} into \cref{eq:eps_stat_first_part} yields
    \begin{equation}
        \label{eq:eps-stat-second-part}
        \riemanniannorm[auto]{
            \sequence{Y}{k+1}
            +
            \frac{1}{\sequence{\lambda}{k}}
            \sequence{Z}{k+1}
        }
        \le
        (
            \lipgrad
            +
            \sigmakmax
        )
        \dist(\sequence{p}{k+1}, \sequence{p}{k})
        ,
    \end{equation}
    The claim follows after substituting $\dist(\sequence{p}{k+1}, \sequence{p}{k}) = \sequence \lambda k \riemanniannorm[auto]{\Gmap{\sequence \lambda k}{\sequence{p}{k}}}$.
\end{proof}
The following is a complexity result for the convergence of the iterates to stationary points.
\begin{theorem}
    \label{thm:complexity}
    Let $g$ and $h$ satisfy \cref{assumption:nonconvex-assumptions} and $\varepsilon > 0$.
    Then \Cref{algorithm:NCRPG} finds a point $\sequence{p}{k}$ such that $\riemanniannorm[auto]{X_{\sequence{p}{k}}} \le \varepsilon$, with $X_{\sequence{p}{k}} \in \partial f(\sequence{p}{k})$, in at most
    \begin{equation}
        \label{eq:complexity-rate}
        k
        =
        \frac{
            (
                \lipgrad
                +
                \sigmakmax
            )^2
            (f(\sequence p 0) - f_{opt})
            t^2
        }{
            \varepsilon^2 M
        }
        -
        1
    \end{equation}
    iterations, where $t$ is chosen as $\lambda$ in the case of a constant stepsize and the initial guess $s$ for the backtracking procedure.
\end{theorem}
\begin{proof}
    Choose $
        n
        \ge
        \frac{
            (
                \lipgrad
                +
                \sigmakmax
            )^2
            (f(\sequence p 0) - f_{opt})
            t^2
        }{
            \varepsilon^2 M
        } - 1
    $.
    Notice that our choice of $t$ ensures $\sequence \lambda k \le t$ holds for all $k$.
    Then, by \cref{lemma:nonincrease-stationarity}, there exists a $k \le n$ such that
    \begin{equation*}
        \riemanniannorm{\Gmap{\sequence \lambda k}{\sequence{p}{k}}}
        \le
        \sqrt{
            \frac{f(\sequence p 0)- f_{opt}}{M(n + 1)}
        }
        \le
        \frac{
            \varepsilon
        }{
            (
                \lipgrad
                +
                \sigmakmax
            )
            t
        }
        .
    \end{equation*}
    Combined with \cref{lemma:bound-on-subdifferential}, we obtain an $X_{\sequence{p}{k}} \in \partial f\big(\sequence{p}{k}\big)$ such that
    \[
        \riemanniannorm[auto]{
            X_{\sequence{p}{k}}
        }
        \le
        (
            \lipgrad
            +
            \sigmakmax
        )
        \sequence \lambda k
        \riemanniannorm{\Gmap{\sequence \lambda k}{\sequence{p}{k}}}
        \le
        \frac{
            \sequence \lambda k
        }{
            t
        }
        \varepsilon
        \le
        \varepsilon
        .\qedhere
    \]%\ignorespacesafterend%
\end{proof}

\subsection{Global convergence to stationary points}

We can now give a global convergence result for \Cref{algorithm:NCRPG}, inspired by \cite[Theorem~3.1]{HuangWei:2023:1}.
\begin{theorem}
    \label{thm:global-convergence}
    Let $g$ and $h$ satisfy \cref{assumption:nonconvex-assumptions}.
    Let $\sequence(){p}{k}$ be the sequence generated by \cref{algorithm:NCRPG} for some starting point $\sequence{p}{0}$ and either a constant stepsize $\sequence{\lambda}{k} \coloneq \bar \lambda \in \left(0, \min\left\{\lambdadelta, \frac{\zetadelta}{\lipgrad}\right\}\right)$ or the backtracking procedure with constants $(s,  \beta, \eta)$.
    Then, the sequence $\sequence(){p}{n}$ has at least one accumulation point, and any accumulation point is a stationary point of $f$.
\end{theorem}
\begin{proof}
    By the compactness of the sublevel set $\startlevelset$ of $f$ and \cref{lemma:sufficient-decrease-with-constant}, the sequence $\sequence(){p}{k}$ has at least one accumulation point.
    Let now $\sequence(){p}{n_j}$ be a subsequence converging to a point $p^\ast \in \startlevelset$.
    Then, by \cref{lemma:bound-on-subdifferential} we find a sequence of tangent vectors $X_{\sequence{p}{n_j}} \in \partial f\big(\sequence{p}{n_j}\big)$ such that
    \begin{equation*}
        \lim_{j \to \infty}
        \riemanniannorm[auto]{X_{\sequence{p}{n_j}}}
        \le
        \lim_{j \to \infty}
        (
            \lipgrad
            +
            \sigmakmax
        )
        \sequence \lambda {n_j}
        \,
        \riemanniannorm[auto]{
            \Gmap{\sequence \lambda {n_j}}{\sequence{p}{n_j}}
        }
        = 0,
    \end{equation*}
    where the last step follows from \cref{lemma:nonincrease-stationarity}.
    \Cref{assumption:nonconvex-assumptions} implies the continuity of $f$.
    Thus, $f\left(\sequence{p}{n_j}\right) \to f(p^\ast)$, and since we found a (sub)sequence $X_{\sequence{p}{n_j}} \in \partial f\left(\sequence{p}{n_j}\right)$ such that $X_{\sequence{p}{n_j}} \to 0_{p^\ast}$,
    we can conclude that $0_{p^\ast} \in \partial f(p^\ast)$ by \cite[Theorem~2.2]{HosseiniHuangYousefpour:2018:1}, implying the stationarity of $p^\ast$.
\end{proof}

\section{A Riemannian proximal gradient method using retractions}
\label{section:Retractions}
Often, we do not have access to a closed form of the exponential map
$\exponential{p}{X}$ or its inverse, the logarithmic map. In other cases, computing these maps is simply too costly. For first order convergence results, it is often enough to consider first order approximations of the exponential map, so called {\em retractions}, and their inverses. In \cite{HuangWei:2021:1}, the authors have shown convergence of their algorithm also for such retractions. This alters the proximal problem on the tangent space but it does not directly affect the algorithm itself. In this section, we will discuss the use of retractions in the NCRPG method and conclude that convergence theory remains largely an open problem.

A retraction on a Riemannian manifold $\cM$ is a first-order approximation of the exponential map
$
    \retractionSymbol
    \colon
    \tangentBundle
    \to
    \cM
    ,
    (p, X_p)
    \mapsto
    \retract{p}(X_p)
    ,
$
such that each curve $c(t) = \retract{p}(tX_p)$ satisfies $c(0) = p$ and ${c'(0) = X_p}$.
Similarly, the inverse retraction $\inverseRetract{p}(q)$ is a first order approximation of the logarithmic map $\logarithm{p} q$.
Note that analogously to the logarithmic map, both the retraction and the inverse retraction might only be locally defined around $p$.%

We can use these approximations to consider both an approximation of the Riemannian distance $\dist$ and the proximal map $\prox{\lambda f}$.
We denote
\begin{equation*}
    \distR(p,q)
    \coloneq
    \riemanniannorm{\inverseRetract{p}(q)}[p]
    \qquad
    \text{and}
    \qquad
    \proxR{\lambda f}(p)
    \coloneq
    \argmin_{q\in\cM}
    \frac{1}{2\lambda}\distR(q,p)^2
    +
    f(q)
    ,
\end{equation*}
for a given retraction $\retractionSymbol$, respectively. We note that the retraction distance $\distR$ is not symmetric but both approximations $\distR(p,q) \approx \dist(p,q) \approx \distR(q,p)$ are equally valid. We have chosen the latter in the retraction-based proximal map, as it leads to nicer expressions.

We can then replace the minimization of \Cref{eq:nonconvex-subproblem-objective} by a retraction-based proximal map.
\begin{align}\label{eq:retraction-prox}
    \proxR[big]{\lambda h}[\retractionSymbol](
        \mathrm{R}_{p}(-\lambda\grad g(p))
    )
    &=
    \argmin_{q\in\cM}
    \frac{1}{2\lambda}
    \Bigl(
        \distR[big](
            q,\mathrm{R}_{p}(-\lambda\grad g(p))
        )[\retractionSymbol]\Bigr)^2
    +
    h(q) \\
    &=
    \argmin_{q\in\cM}
    \frac{1}{2\lambda}
    \riemanniannorm[auto]{
        \inverseRetract{q}\Bigl(\mathrm{R}_{p}\bigl(-\lambda\grad g(p)\bigr)\Bigr)
    }^2
    +
    h(q)
    .
\end{align}
Or, as above, by finding a stationary point $p^\ast$ of the objective function
\begin{equation}\label{equation:retraction-objective}
    H^\retractionSymbol(q;
      p,
      \lambda
    )
    \coloneq
    h(q)
    +
    \frac{1}{2 \lambda}
    \left(
        \distR[Big](
            q,
            \retract{p}\left(
                -\lambda \grad g(p)\right)
        )[\retractionSymbol]
    \right)^2,
\end{equation}
which fulfills $H^\retractionSymbol(p^\ast; p, \lambda)
\leq
H^\retractionSymbol(p; p, \lambda).$
This is again different from \cite{HuangWei:2021:1} in that the problem is posed on the manifold as opposed to a tangent space.
We then summarize the retraction-based algorithm in \Cref{algorithm:NCRPGretr}.

The main problem with obtaining a convergence theory of this method - at least using the considerations above - is that \cref{lemma:sufficient-decrease} only holds for the exact Riemannian proximal gradient method, where steps are taken along geodesics.
Generalizing this result to approximations thereof, where the steps are taken using retractions, will add one of the following difficulties:
\begin{algorithm}[htp]
	\caption{Retraction-based Nonconvex Proximal Gradient Method}
	\label{algorithm:NCRPGretr}
	\begin{algorithmic}[1]
		\Require
        $g$, %
        $\grad g$, %
        $h$, %
        retraction $\retractionSymbol$, %
        a sequence $\sequence{\lambda}{k}$, %
        initial point $\sequence{p}{0} \in \cM$.
		\While{convergence criterion is not fulfilled}

    \State Find a stationary point $\sequence{p}{k+1}$ of $H^\retractionSymbol(\cdot; \sequence{p}{k}, \sequence{\lambda}{k})$
    \Statex \hspace{0.45cm} \suto~$H^\retractionSymbol(\sequence{p}{k+1}; \sequence{p}{k}, \sequence{\lambda}{k})
    		\le H^\retractionSymbol(\sequence{p}{k};
		    \sequence{p}{k},
		    \sequence{\lambda}{k}
		)$.

		\State Set $k \coloneq k+1$
		\EndWhile
	\end{algorithmic}
\end{algorithm}
Either it would be necessary to generalize the law of cosines for (some special) retractions.
So far, this seems out of reach.
Alternatively, it could be possible to rewrite the update steps in terms of geodesic distances and then apply the standard law of cosines.
One way to do this requires \cref{eq:retraction-property} below.
We keep the modified proximal objective function $H^\retractionSymbol$
but the $\lipgrad$ smoothness in \Cref{assumption:nonconvex-assumptions:g} is still stated in terms of logarithm and geodesic distance. That is,
\begin{equation}
    g(p)
    \ge
    g(q)
    - \riemannian{\grad g(p)}{\logarithm{p}q}
    - \frac{\lipgrad}{2} \dist^2 \left(
        p
        ,
        q
    \right).
\end{equation}
Combining these, we now get
\begin{align*}
    f(p)
    &
    \ge
    f(\Imap[auto]{\lambda}{p})
    +
    \frac{1}{\lambda}
    \riemannian{
        \logarithm[big]{p}({\exponential{p}(- \lambda \grad g(p))})
    }{
        \logarithm{p} \Imap[auto]{\lambda}{p}
    }
    -
    \frac{\lipgrad}{2} \dist^2 \left(
        p
        ,
        \Imap[auto]{\lambda}{p}
    \right)
    \\
    &
    \quad
    +
    \frac{1}{2 \lambda} \distR[auto](
        \Imap[auto]{\lambda}{p}
        ,
        \retract{p}(-\lambda \grad g(p))
    )[\retractionSymbol]^2
    -
    \frac{1}{2 \lambda} \distR[auto]\left(
        p
        ,
        \retract{p}(-\lambda \grad g(p))\right)^2
\end{align*}
Define now $q \coloneq \Imap[auto]{\lambda}{p}$, and $\gradientstep{p} \coloneq \exponential{p}(-\lambda \grad g(p))$, and get
\begin{align*}
    f(p)
    &
    \ge
    f(q)
    + \frac 1 {\lambda}
    \riemannian{\logarithm{p}(\gradientstep{p})}
        {\logarithm{p}(q)}
    -
    \frac{\lipgrad}{2} \dist^2 \left(
        p
        ,
        q
    \right)
    \\
    &
    \quad
    +
    \frac{1}{2 \lambda} \riemanniannorm[auto]{
        \inverseRetract{q}{\retract{p}(-\lambda \grad g(p))}
    }^2
    -
    \frac{1}{2 \lambda} \dist^2\left(
        p
        ,
        \gradientstep{p}\right)
\end{align*}

At this point, if there is a constant $c \in \bbR$ such that for all $p, q \in \cM$ and $v \in \tangentSpace{p}$,
\begin{equation}\label{eq:retraction-property}
    \riemanniannorm[auto]{
        \inverseRetract{q}{\retract{p}(v)}
    }^2
    \ge
    \riemanniannorm[auto]{
        \logarithm{q}{\exponential{p}(v)}
    }^2
    - c \dist^2 \left(p,q \right)
\end{equation}
holds, we obtain
\begin{align*}
    f(p)
    &
    \ge
    f(q)
    + \frac 1 {\lambda}
    \riemannian{\logarithm{p}{\gradientstep{p}}}
        {\logarithm{p}{q}}
    -
    \frac{\tilde L_g}{2} \dist^2 \left(
        p
        ,
        q
    \right)
    +
    \frac{1}{2 \lambda} \dist^2 \left(
        \gradientstep{p}
        ,
        q
    \right)
    -
    \frac{1}{2 \lambda} \dist^2\left(
        p
        ,
        \gradientstep{p}\right),
\end{align*}
where $\tilde L_g = \lipgrad + \frac c \lambda$.
The rest of the proof would now be the same as before.

The requirement in \Cref{eq:retraction-property} can be seen as a first-order approximation property of the exponential (and logarithmic map) \emph{in the first, manifold argument}, instead of the second, tangent space argument.
Whether we can find retractions that satisfy this requirement, or if that is even always possible, is out of the scope of this work and we will leave it, as well as the whole convergence theory of \Cref{algorithm:NCRPGretr}, as an open problem.

\section{Numerical Examples}
\label{section:numerics}

In this section, we present numerical experiments to evaluate the performance of the Nonconvex Riemannian Proximal Gradient (NCRPG) method described in \cref{algorithm:NCRPG}. 
We consider three applications on manifolds that are non-Hadamard (\ie, they do not exhibit nonpositive curvature everywhere).
The NCRPG method has been implemented in \julia 1.11 within \manoptjl 0.5.
The implementation is available with the level of generality provided by retractions as well, as discussed in \cref{section:Retractions}.
The numerical experiments were performed on a MacBook~Pro~M1, 16~GB~RAM running macOS~26.5.
The experiments are available as notebooks in the package \manoptexamplesjl \cite{BergmannJasa:2025:17277311}.

\subsection{Sparse PCA}\label{subsection:numerics-problem-setting-SPCA}
The first example is Sparse PCA. Here, we consider the following optimization problem:
\begin{equation}
    \min_{X \in \text{OB}(n,r)} f(X)
    =
     \frac 1 2
    \| X^\top A^\top A X - D^2 \|^2
    +
    \mu \norm{X}_1
    ,
\end{equation}
where $A \in \bbR^{m\times n}$ is a given data matrix, $D$ contains the $r$ leading singular values of $A$ on the diagonal, $\mu \ge 0$ is a weighting parameter, and $\text{OB}(n, r) = \bigotimes_{i = 1}^r \mathbb S^{n-1}$ is the {\em Oblique manifold} of matrices with normalized columns. Hence, the goal to diagonalize $A^\top A$ as much as possible using normalized components that are forced to be sparse by the $\ell_1$-penalty.
This variation of the problem was introduced in \cite{GenicotHuangTrendfilov:2015:1}.
We note that there are alternative formulations of the Sparse PCA problem on the Stiefel manifold that we do not consider here.

\paragraph{Soft Thresholding on the Sphere}\label{subsubsection:Soft-Thresholding-Sphere}
The main advantage of our algorithm in this case is that the prox-problem can be efficiently solved using a fixed-point iteration, as we will show in the following. Since the Oblique manifold is a product of spheres, the prox-problem decomposes and it is sufficient to solve it on the sphere. Thus, we aim to solve
\begin{equation}\label{eq:soft-thresholding-on-the-sphere:problem}
    \argmin_{y \in \bbS^{n-1}} \phi(y) \coloneq
    \frac{1}{2}
    \dist(x, y)^2
    +
    \xi \norm{y}_1
    ,
\end{equation}
where $x \in \bbS^{n-1}$, $\xi \ge 0$, and $\dist(x, y) = \arccos(\langle x, y\rangle)$ is the Riemannian distance.

We will now show that the minimum $\bar y = \argmin_{y \in \bbS^{n-1}} \phi(y)$ is attained at $\bar y = p_x(\bar t)$ and $\bar t$ is a fixed point $\bar t = \sigma(p_x(\bar t))$ where
\begin{equation}
    \label{eq:def-sigma}
    \sigma(y)
    \coloneq
    \xi \frac{
        \sqrt {1 - \langle x ,{y}\rangle ^2}
    }{
        \arccos \left(
            \langle x, y\rangle
        \right)
    }
    , \qquad \text{and} \qquad
    p_x(t) \coloneq \frac {\prox{ t \norm{\cdot}_1}^{\bbR^{n+1}}(x)}{\norm{\prox{ t \norm{\cdot}_1}^{\bbR^{n+1}}(x)}_2}
\end{equation}
is the normalized Euclidean proximal operator. Then, $\bar t$, and with that also $\bar y$ can be computed in linear time using Banach's fixed point theorem.

Notice that, since $\phi$ is continuous and $\bbS^{n-1}$ is compact, such a $\bar y$ exists, and $0_{\bar y} \in \partial \phi(\bar y)$.
Hence,
\begin{equation}
    \label{eq:soft-thresholding-on-the-sphere:zero-subgrad}
    \begin{aligned}
        0_{\bar y}
        &
        =
        - \logarithm{\bar y} x
        +
        \xi \proj{\tangentSpace{\bar y}[\bbS^{n-1}]}(z)
        \\
        &
        =
        -\frac{
            \arccos \left(\langle x, {\bar y}\rangle \right)
        }{
            \Vert\proj{\tangentSpace{\bar y}[\bbS^{n-1}]}(x)\Vert_2
        }
        \proj{\tangentSpace{\bar y}[\bbS^{n-1}]}(x)
        +
        \xi \proj{\tangentSpace{\bar y}[\bbS^{n-1}]}(z)
        ,
    \end{aligned}
\end{equation}
for some $z \in \partial \norm{\bar y}_1$.
With the projection onto the tangent space at $\bar y$ given by
\begin{equation*}
    \proj{\tangentSpace{\bar y}[\bbS^{n-1}]}(x)
    =
    x
    - \langle x, \bar y \rangle \bar y
    ,
\end{equation*}
we see that \cref{eq:soft-thresholding-on-the-sphere:zero-subgrad} simplifies to $0_{\bar y} = \proj{\tangentSpace{\bar y}[\bbS^{n-1}]}(x-\sigma(y) z)$, which is equivalent to
\begin{equation}
    \label{eq:soft-thresholding-on-the-sphere:zero-subgrad-simplified}
    x
    -
    \sigma(\bar y) z
    =
    \langle{x - \sigma(\bar y) z},{\bar y}\rangle \bar y
    .
\end{equation}
The solution of the proximal operator of the function $t \norm{\cdot}_1$ in Euclidean space $\bbR^{n+1}$ for a point $x\in \bbR^{n+1}$ is given elementwise by
\begin{equation}
    \left(
        \prox{t \norm{\cdot}_1}^{\bbR^{n+1}}(x)
    \right)_i
    =
    \sgn(x_i)
    \max \left\{
        0
        ,
        \abs[auto]{x_i}
        -
        t
    \right\}
    ,
    \qquad
    i
    =
    1
    ,
    \dots
    ,
    n+1
    ,
\end{equation}
see, \eg, \cite[Example~6.8]{Beck:2017:1}.

The following theorem shows that, given $x \in \bbS^{n-1}$, the solution of \cref{eq:soft-thresholding-on-the-sphere:problem} can be obtained by applying the normalized Euclidean proximal operator of $t \norm{\cdot}_1$ to $x$, for a particular choice of $t$.
\begin{theorem}\label{theorem:prox-l1-sphere}
    Let $x\in \bbS^{n-1}$ and let $\bar y$ be a solution of the optimization problem \cref{eq:soft-thresholding-on-the-sphere:problem} with $0 < \xi < \norm{x}_\infty$.
    Then $\bar y = p_x(\bar t)$ with $\bar t$ being a fixed point of $\sigma(p_x(t))$.
\end{theorem}
\begin{proof}
    Let $\bar y$ be a solution of \cref{eq:soft-thresholding-on-the-sphere:problem} and $i \in \{1,\ldots,n\}$ be a fixed index. We first show that, if $x_i > 0$ ($x_i < 0$), then $\bar y_i \ge 0$ ($\bar y_i \le 0$):
    Assume $x_i > 0$ and $\bar y_i < 0$.
    Define $\hat y$ as $\hat y_i \coloneq - \bar y_i$ and $\hat y_j \coloneq \bar y_j$ for $j \ne i$.
    Then $\hat y \in \bbS^{n-1}$ and $\norm {\hat y}_1 = \norm {\bar y}_1$.
    But $\langle x, {\hat y}\rangle = \langle{x},{\bar y}\rangle  - 2 x_i \bar y_i > \langle {x}, {\bar y}\rangle$,
    hence $\arccos \left(\langle{x},{\hat y}\rangle \right) < \arccos \left(\langle{x},{\bar y} \rangle\right)$, implying $\phi(\hat y) < \phi({\bar y})$, contradicting the optimality of $\bar y$. 
    The same argument holds for $x_i < 0$ and $\bar y_i > 0$.

    Let $z \in \partial \norm {\bar y}_1$ be such that \cref{eq:soft-thresholding-on-the-sphere:zero-subgrad-simplified} holds.
    Notice that, since $\langle{\bar y},{x}\rangle \in [0, 1]$, we have $\sigma(\bar y) \in \left[\frac {2\xi} \pi, \xi\right]$.
    W.l.o.g.~we can assume that $x_1 = \norm{x}_\infty$.
    Since $\xi < x_1$ and $\sigma(\bar y) \le \xi$, we have $x_1 - \sigma(\bar y) z_1 > 0 $.
    Hence, by \cref{eq:soft-thresholding-on-the-sphere:zero-subgrad-simplified}, $\bar y_1 \ne 0$.
    Combined with $\bar y_1 \ge 0$ we have
    \begin{equation*}
        \langle{
            x
            -
            \sigma(\bar y) z
        },{\bar y}
        \rangle
        =
        \frac{
            x_1
            -
            \sigma(\bar y) z_1
        }{\bar y_1}
        >
        0
        .
    \end{equation*}
    Let now $i \in \{2, \ldots, n\}$.
    We first consider the case $\abs[auto]{x_i} \le \sigma(\bar y)$.
    Assume that $\bar y_i \ne 0$, then ${\sgn(x_i - \sigma(\bar y) z_i) \in \{-\sgn(\bar y_i), 0\}}$.
    But then by \cref{eq:soft-thresholding-on-the-sphere:zero-subgrad-simplified}
    \begin{equation*}
        \sgn(\bar y_i)
        =
        \sgn(
            \langle{
                x
                -
                \sigma(\bar y) z
            },{\bar y}
            \rangle
        )
        \sgn(\bar y_i)
        =
        \sgn(
            x_i
            -
            \sigma(\bar y) z_i
        )
        \in
        \{
            -\sgn(\bar y_i)
            ,
            0
        \}
        ,
    \end{equation*} which is a contradiction.
    Thus $\bar y_i = 0$ and $z_i = \frac{x_i}{\sigma(\bar y)}$ must hold.

    If $\abs[auto]{x_i} > \sigma(\bar y)$ we have $\sgn(x_i - \sigma(\bar y) z_i) = \sgn(x_i)$, implying $\bar y_i \ne 0$ by \cref{eq:soft-thresholding-on-the-sphere:zero-subgrad-simplified}.
    In summary, we obtain
    \begin{equation*}
        x_i
        -
        \sigma(\bar y) z_i
        =
        \begin{cases}
            0
            &
            \text{ if }
            \abs[auto]{x_i}
            \le
            \sigma(\bar y)
            ,
            \\
            x_i
            -
            \sigma(\bar y) \sgn(x_i)
            &
            \text{ otherwise}
            ,
        \end{cases}
    \end{equation*}
    \ie, $x - \sigma(\bar y) z = \prox{\sigma(\bar y) \norm{\cdot}_1}^{\bbR^{n+1}}(x)$.
    Now, since $\norm{\bar y}_2 = 1$ and $\langle{x - \sigma(\bar y) z},{\bar y}\rangle > 0$, we have
    \begin{equation}
        \bar y
        =
        \frac{
            x
            -
            \sigma(\bar y) z
        }{
            \norm{
                x
                -
                \sigma(\bar y) z
            }_2
        }
        =
        p_x(\sigma(\bar y))
        .
    \end{equation}
    In particular this implies that $\bar t = \sigma (\bar y)$ is a fixed point of $\sigma(p_x(t))$.
\end{proof}

With the above theorem, we can conclude that we are looking for a fixed point of $\sigma(p_x(t))$. We will now show that this fixed point can be computed using Banach's fixed-point theorem.

\begin{lemma}\label{lemma:fixed-point-sol-prox-l1-sphere}
    Let $x \in \bbS^{n-1}$ and
    \begin{equation}
        \label{eq:mu-upper-bound}
        0
        <
        \xi
        <
        \min\left\{
            \frac{\pi}{8 \sqrt{\norm x_0}}
            \left(
                \sqrt{
                    16 \sqrt{\norm x_0}
                    \norm{x}_\infty
                    +
                    \pi^2
                }
                -
                \pi
            \right)
            ,
            \norm{x}_\infty
        \right\}
        .
    \end{equation}
    Then $\sigma(p_x(t)) \colon (0, \xi] \to (0,\xi]$ has a unique fixed point $\bar t$, and for every $t_0 \in (0, \xi]$ the sequence ${t_{k+1} \coloneq \sigma(p_x(t_k))}$ converges linearly to $\bar t$.
\end{lemma}
\begin{proof}
    Notice that, by our choice of $\xi$, $p_x(t)$ is well-defined for all $t \in (0, \xi]$ and $\sigma(p_x(t)) \in (0, \xi]$.
    We now show that $p_x(t)$ is a contraction mapping.
    Define $\psi(s) \coloneq \xi \frac{\sqrt {1- s^2}}{\arccos(s)}$ for $s \in [0, 1]$.
    We then have
    \begin{equation*}
        \sup_{s \in [0, 1)}
        \psi'(s)
        =
        \sup_{s \in [0, 1)}
        \xi \frac{
            1
            -
            \frac{s \arccos(s)}{\sqrt{1 - s^2}}
        }{\arccos(s)^2}
        =
        \frac{4 \xi}{\pi^2}
        .
    \end{equation*}
    Thus, for any $t_1, t_2 \in (0, \xi]$, we obtain
    \begin{align*}
        \abs[auto]{
            \sigma(p_x(t_1))
            -
            \sigma(p_x(t_2))
        }
        &
        =
        \abs[auto]{
            \psi(\langle{p_x(t_1)},{x}\rangle)
            -
            \psi(\langle{p_x(t_2)},{x}\rangle)
        }
        \\
        &
        \le
        \frac{4 \xi}{\pi^2}
        \abs[auto]{
            \langle{p_x(t_1)},{x}\rangle
            -
            \langle{p_x(t_2)},{x}\rangle
        }
        .
    \end{align*}
    Since $p_x(t), x \in \bbS^{n-1}$, we have
    \begin{equation*}
        \langle{p_x(t)},{x}\rangle
        =
        1
        -
        \frac{1}{2}
        \norm{p_x(t) - x}_2^2
        .
    \end{equation*}
    Hence, for any $t_1, t_2 \in (0, \xi]$, we rewrite
    \begin{align*}
        \abs[auto]{
            \langle{p_x(t_1)},{x}\rangle
            -
            \langle{p_x(t_2)},{x}\rangle
        }
        &
        =
        \abs[auto]{
            \frac{1}{2}
            \norm{p_x(t_1) - x}_2^2
            -
            \frac{1}{2}
            \norm{p_x(t_2) - x}_2^2
        }
        .
    \end{align*}
    W.l.o.g., $\frac{1}{2} \norm{p_x(t_1) - x}_2^2 \ge \frac{1}{2} \norm{p_x(t_2) - x}_2^2$.
    We now use that $p_x(t_1)$ is a solution of
    \begin{equation*}
        \argmin_{y\in \bbS^{n-1}}
        \quad
        \frac{1}{2}
        \norm{x - y}_2^2
        +
        t_1\norm y_1
        .
    \end{equation*}
    A proof can be found in \cite[D.1]{HuangWei:2023:1}.
    Therefore, we have
    \begin{align*}
        \abs[auto]{
            \langle{p_x(t_1)},{x}\rangle
            -
            \langle{p_x(t_2)},{x}\rangle
        }
        &
        =
        \frac{1}{2}
        \norm{p_x(t_1) - x}_2^2
        -
        \frac{1}{2}
        \norm{p_x(t_2) - x}_2^2
        \\
        &
        \le
        t_1 \left(
            \norm{p_x(t_2)}_1
            -
            \norm{p_x(t_1)}_1
        \right)
        \\
        &
        \le
        t_1 \norm{p_x(t_2) - p_x(t_1)}_2
        ,
    \end{align*}
    where we used the Lipschitz continuity of the $\ell_1$-norm with constant $1$ for the last step.

    Since $\norm{\prox{t \norm{\cdot}_1}^{\bbR^{n+1}}(x)}_2 \ge \norm x_\infty - \xi$ for $t \le \xi$, we get
    \begin{align*}
        \norm{p_x(t_2) - p_x(t_1)}_2
        &
        \le
        \frac{1}{\norm x_\infty - \xi}
        \norm{
            \prox{
                t_1 \norm{\cdot}_1
            }^{\bbR^{n+1}}(x)
            -
            \prox{
                t_2 \norm{\cdot}_1
            }^{\bbR^{n+1}}(x)
        }
        \\
        &
        =
        \frac{1}{\norm x_\infty - \xi}
        \norm{
            \prox{
                t_1 \norm{\cdot}_1
            }^{\bbR^{n+1}}(x)
            -
            \prox{
                t_1 \norm{\cdot}_1
            }^{\bbR^{n+1}}(
                x
                +
                \bar x
                (t_1 - t_2)
            )
        }_2
        \\
        &
        \le
        \frac{1}{\norm x_\infty - \xi}
        \norm{
            x
            -
            (
                x
                +
                \bar x
                (t_1 - t_2)
            )
        }_2
        \\
        &
        =
        \frac{\sqrt{\norm x_0}}{\norm x_\infty - \xi}
        (t_1 - t_2)
        ,
    \end{align*}
    where $\bar x$ is defined as $\bar x_i \coloneq \sgn(x_i)$ and the second to last step uses the non-expansivity of the proximal operator of the $\ell_1$-norm, see, \eg, \cite[Thm.~6.42]{Beck:2017:1}.
    All in all we get
    \begin{align*}
        \abs[auto]{
            \sigma(p_x(t_1))
            -
            \sigma(p_x(t_2))
        }
        &
        \le
        \frac{
            4 \xi t_1 \sqrt{\norm x_0}
        }{(\norm x_\infty - \xi) \pi^2}(t_1-t_2)
        \\
        &
        \le
        \frac{
            4 \xi^2 \sqrt{\norm x_0}
        }{(\norm x_\infty - \xi) \pi^2}(t_1-t_2)
        .
    \end{align*}
    The upper bound on $\xi$ in \cref{eq:mu-upper-bound} ensures $\frac{4 \xi^2 \sqrt{\norm x_0}}{(\norm x_\infty - \xi) \pi^2} < 1$.
    Applying the Banach fixed-point theorem yields the stated result.
\end{proof}

\Cref{lemma:fixed-point-sol-prox-l1-sphere} implies that, if additionally
\begin{equation}
    \xi
    <
    \frac{\pi}{8 \sqrt{\norm x_0}}
    \left(
        \sqrt{
            16 \sqrt{\norm x_0}
            \norm{x}_\infty
            +
            \pi^2
        }
        -
        \pi
    \right)
\end{equation}
holds, the solution of \cref{eq:soft-thresholding-on-the-sphere:problem} is unique and can be computed in linear time.
Since our method solves the proximal mapping with the parameter $\xi = \sequence \lambda k \mu$, this bound is never exceeded in practice.%

\paragraph{\textbf{Implementation Details}}%
\label{subsection:numerics-implementation-SPCA}%
\footnote{The code is available at \url{https://juliamanifolds.github.io/ManoptExamples.jl/stable/examples/NCRPG-Sparse-PCA/}}
We tested the NCRPG method using a data matrix $A \in \mathbb R^{m \times n}$, where $m = 20$, that was generated with normally distributed random entries.
Then, the columns were centered and normalized by their standard deviation.

We ran our method as in \Cref{algorithm:NCRPG} using the above fixed point iteration for the computation of the prox problem.
As a comparison, we tested the Riemannian Proximal Gradient method (RPG) from \cite{HuangWei:2021:1} and the Manifold Proximal Gradient method (ManPG) \cite{ChenMaMan-ChoSoZhan:2020:1}.
The RPG and ManPG solve a different proximal problem stated over the tangent space.
At each iteration, the RPG aims to find a stationary point $\eta^*_{\sequence{x}{k}} \in \tangentSpace{\sequence{x}{k}}[\text{OB}(n,r)]$ of the function
\begin{equation}
    \ell_{\sequence{x}{k}}(\eta)
    =
    \riemannian[auto]
    {\grad g\left(\sequence{x}{k}\right)}
    {\eta}
    +
    \frac{1}{2\lambda} \norm{\eta}^2
    +
    h(\retract{\sequence{x}{k}}(\eta))
    ,
\end{equation}
s.t. $\ell_{\sequence{x}{k}}(\eta^*_{\sequence{x}{k}}) \leq \ell_{\sequence{x}{k}}(0) $.
Although the problem is stated over a linear space, it remains hard to solve in most cases.
However, for the sparse PCA problem stated over the oblique manifold, an explicit iterative algorithm is given in \cite[Section~5.2]{HuangWei:2021:1} when the exponential map is used as the retraction.
ManPG computes a minimizer $\eta^*_{\sequence{x}{k}} \in \tangentSpace{\sequence{x}{k}}[\text{OB}(n,r)]$ of $\ell_{\sequence{x}{k}}$ at each iteration but relaxes the retraction to the addition in the embedding space. Additionally, a backtracking procedure with contraction factor $\beta = 0.5$ is applied to compute the next iterate.
As suggested in \cite[Section~4.2]{ChenMaMan-ChoSoZhan:2020:1}, the subproblem is solved using a regularized semi-smooth Newton method (SSN).

We use $\lambda = \frac{1}{2 \norm{A}^2_F}$ as a constant stepsize for NCRPG, RPG and ManPG and an initial guess $s = \frac 5 {\norm{A}^2_F}$ and parameters $\beta = \eta = 0.5$ in \cref{algorithm:backtracking} for NCRPG with backtracking.
Each run was initialized with a random starting point using the uniform distribution on the oblique manifold.
We stopped NCRPG when the norm of the gradient mapping fell below $10^{-4}$ with the fixed point iteration described above terminating when either $\vert \sequence{t}{k+1} - \sequence{t}{k} \vert < 10^{-10}$ or after $10$ iterations (usually after only 2-3 iterations).
Similarly, RPG and ManPG were stopped when $\norm{\frac 1 \lambda \eta^*_{\sequence{x}{k}}} < 10^{-4}$.
The SSN subsolver of ManPG was terminated when either the residuum fell below $10^{-12}$ or the number of iterations exceeded $10$.
Additionally, all algorithms were terminated after exceeding $100\,000$ iterations.

\begin{figure}[tbp]
    \begin{subfigure}{.5\textwidth}
      \centering
      \begin{tikzpicture}
        \begin{semilogyaxis}[
            width=0.99\textwidth,
            xlabel={time (s)},
            ylabel={function value error},
            ylabel near ticks,
            xlabel near ticks,
            ymajorgrids,
            label style={font=\scriptsize},
            tick label style={font=\scriptsize},
            legend style={font=\scriptsize, nodes={scale=0.66}},
            legend pos=north east,
            legend cell align=left,
        ]
        \addplot[color=TolBrightBlue,thick] table {data/plot_oblique/diff_fs_NCRPG.csv};
        \addplot[color=TolBrightCyan,thick] table {data/plot_oblique/diff_fs_NCRPG_bt.csv};
        \addplot[color=TolBrightRed,thick] table {data/plot_oblique/diff_fs_RPG.csv};
        \addplot[color=TolBrightYellow,thick] table {data/plot_oblique/diff_fs_ManPG.csv};
        \legend{NCRPG constant stepsize, NCRPG backtracking, RPG, ManPG}
        \end{semilogyaxis}
      \end{tikzpicture}
      \caption{Error of the cost function $f(\sequence{X}{k}) - f_{opt}$.}
      \label{subfig:oblique-cost-error}
      \end{subfigure}%
  \begin{subfigure}{0.5\textwidth}
      \centering
    \begin{tikzpicture}
    \begin{semilogyaxis}[
        width=0.99\textwidth,
        xlabel={time (s)},
        ylabel={gradient mapping norm},
        ylabel near ticks,
        xlabel near ticks,
        ymajorgrids,
        label style={font=\scriptsize},
        tick label style={font=\scriptsize},
        legend style={font=\scriptsize, nodes={scale=0.66}},
        legend pos=north east,
        legend cell align=left,
    ]
    \addplot[color=TolBrightBlue,thick] table {data/plot_oblique/grad_maps_NCRPG.csv};
    \addplot[color=TolBrightCyan,thick] table {data/plot_oblique/grad_maps_NCRPG_bt.csv};
    \addplot[color=TolBrightRed,thick] table {data/plot_oblique/grad_maps_RPG.csv};
    \addplot[color=TolBrightYellow,thick] table {data/plot_oblique/grad_maps_ManPG.csv};
    \legend{NCRPG constant stepsize, NCRPG backtracking, RPG, ManPG}
    \end{semilogyaxis}
    \end{tikzpicture}
    \caption{Norm of the gradient mapping.}
    \label{subfig:oblique-grad-map-norm}
  \end{subfigure}%
\caption{Comparison of NCRPG, RPG and ManPG for Sparse PCA on $\text{OB}(n,r)$ with $n = 100$, $r = 5$ and sparsity parameter $\mu = 0.5$. $f_{opt}$ was computed using the NCRPG method with a constant stepsize, terminated when the gradient mapping norm dropped below $10^{-8}$. For RPG, we used ${\tilde L\sequence \eta k}$ as gradient mapping.}
\label{fig:oblique-comparison-time}
\end{figure}

\begin{figure}[tbp]
    \pgfplotstableread[col sep=comma]{data/oblique/results-OB-obj-spar-orth-10.csv}\objspar
    \pgfplotstableread[col sep=comma]{data/oblique/results-OB-time-iter-10.csv}\timeiter
    % Select table rows between indices
    \pgfplotsset{
      select coords between index/.style 2 args={
        x filter/.code={
          \ifnum\coordindex<#1\def\pgfmathresult{}\fi
          \ifnum\coordindex>#2\def\pgfmathresult{}\fi
        }
      }
    }

    \centering
    \begin{tikzpicture}
        \begin{groupplot}[
            group style={
                group size=3 by 3,
                horizontal sep=1cm,
                vertical sep=1.2cm
            },
            width=0.38\textwidth,
            label style={font=\tiny},
            tick label style={font=\tiny},
            legend style={
                at={(0.98,0.02)},
                anchor=south east,
                font=\tiny,
                row sep=-2pt,
                inner sep=0pt
            },
            legend image post style={scale=0.8},
            legend cell align=left,
            x label style={yshift=8pt},
            y label style={yshift=-18pt},
            title style={yshift=-6pt},
            ymajorgrids,
            %xlabel near ticks,
            %ylabel near ticks
        ]

            % ---- Row 1, mu = 0.1
            \nextgroupplot[xlabel={\tiny $n$}, ylabel={\tiny iterations}]
            \addplot [select coords between index={0}{2}, color=TolBrightBlue, mark=o, mark size=3pt, thick] table[x=n, y=NCRPG_iter] {\timeiter};
            \addplot [select coords between index={0}{2}, color=TolBrightCyan, mark=triangle, mark size=3pt, thick] table[x=n, y=NCRPG_bt_iter] {\timeiter};
            \addplot [select coords between index={0}{2}, color=TolBrightRed, mark=star, mark size=3pt, thick] table[x=n, y=RPG_iter] {\timeiter};
            \addplot [select coords between index={0}{2}, color=TolBrightYellow, mark=x, mark size=3pt, thick] table[x=n, y=ManPG_iter] {\timeiter};

            \nextgroupplot[xlabel={\tiny $n$}, ylabel={\tiny time (s)},title={$\mu = 0.1$}]
            \addplot [select coords between index={0}{2}, color=TolBrightBlue, mark=o, mark size=3pt, thick] table[x=n, y=NCRPG_time] {\timeiter};
            \addplot [select coords between index={0}{2}, color=TolBrightCyan, mark=triangle, mark size=3pt, thick] table[x=n, y=NCRPG_bt_time] {\timeiter};
            \addplot [select coords between index={0}{2}, color=TolBrightRed, mark=star, mark size=3pt, thick] table[x=n, y=RPG_time] {\timeiter};
            \addplot [select coords between index={0}{2}, color=TolBrightYellow, mark=x, mark size=3pt, thick] table[x=n, y=ManPG_time] {\timeiter};

            \nextgroupplot[xlabel={\tiny $n$}, ylabel={\tiny sparsity}, ymin=0, ymax=1]
            \addplot [select coords between index={0}{2}, color=TolBrightBlue, mark=o, mark size=3pt, thick] table[x=n, y=NCRPG_sparsity] {\objspar};
            \addplot [select coords between index={0}{2}, color=TolBrightCyan, mark=triangle, mark size=3pt, thick] table[x=n, y=NCRPG_bt_sparsity] {\objspar};
            \addplot [select coords between index={0}{2}, color=TolBrightRed, mark=star, mark size=3pt, thick] table[x=n, y=RPG_sparsity] {\objspar};
            \addplot [select coords between index={0}{2}, color=TolBrightYellow, mark=x, mark size=3pt, thick] table[x=n, y=ManPG_sparsity] {\objspar};

            \legend{NCRPG, NCRPG bt, RPG, ManPG}

            % ---- Row 2 ----
            \nextgroupplot[xlabel={\tiny $n$}, ylabel={\tiny iterations}]
            \addplot [select coords between index={3}{5}, color=TolBrightBlue, mark=o, mark size=3pt, thick] table[x=n, y=NCRPG_iter] {\timeiter};
            \addplot [select coords between index={3}{5}, color=TolBrightCyan, mark=triangle, mark size=3pt, thick] table[x=n, y=NCRPG_bt_iter] {\timeiter};
            \addplot [select coords between index={3}{5}, color=TolBrightRed, mark=star, mark size=3pt, thick] table[x=n, y=RPG_iter] {\timeiter};
            \addplot [select coords between index={3}{5}, color=TolBrightYellow, mark=x, mark size=3pt, thick] table[x=n, y=ManPG_iter] {\timeiter};

            \nextgroupplot[xlabel={\tiny $n$}, ylabel={\tiny time (s)},title={$\mu = 0.5$}]
            \addplot [select coords between index={3}{5}, color=TolBrightBlue, mark=o, mark size=3pt, thick] table[x=n, y=NCRPG_time] {\timeiter};
            \addplot [select coords between index={3}{5}, color=TolBrightCyan, mark=triangle, mark size=3pt, thick] table[x=n, y=NCRPG_bt_time] {\timeiter};
            \addplot [select coords between index={3}{5}, color=TolBrightRed, mark=star, mark size=3pt, thick] table[x=n, y=RPG_time] {\timeiter};
            \addplot [select coords between index={3}{5}, color=TolBrightYellow, mark=x, mark size=3pt, thick] table[x=n, y=ManPG_time] {\timeiter};

            \nextgroupplot[xlabel={\tiny $n$}, ylabel={\tiny sparsity}, ymin=0, ymax=1]
            \addplot [select coords between index={3}{5}, color=TolBrightBlue, mark=o, mark size=3pt, thick] table[x=n, y=NCRPG_sparsity] {\objspar};
            \addplot [select coords between index={3}{5}, color=TolBrightCyan, mark=triangle, mark size=3pt, thick] table[x=n, y=NCRPG_bt_sparsity] {\objspar};
            \addplot [select coords between index={3}{5}, color=TolBrightRed, mark=star, mark size=3pt, thick] table[x=n, y=RPG_sparsity] {\objspar};
            \addplot [select coords between index={3}{5}, color=TolBrightYellow, mark=x, mark size=3pt, thick] table[x=n, y=ManPG_sparsity] {\objspar};

            \legend{NCRPG, NCRPG bt, RPG, ManPG}

            % ---- Row 3 ----
            \nextgroupplot[xlabel={$n$}, ylabel={iterations}]
            \addplot [select coords between index={6}{8}, color=TolBrightBlue, mark=o, mark size=3pt, thick] table[x=n, y=NCRPG_iter] {\timeiter};
            \addplot [select coords between index={6}{8}, color=TolBrightCyan, mark=triangle, mark size=3pt, thick] table[x=n, y=NCRPG_bt_iter] {\timeiter};
            \addplot [select coords between index={6}{8}, color=TolBrightRed, mark=star, mark size=3pt, thick] table[x=n, y=RPG_iter] {\timeiter};
            \addplot [select coords between index={6}{8}, color=TolBrightYellow, mark=x, mark size=3pt, thick] table[x=n, y=ManPG_iter] {\timeiter};

            \nextgroupplot[xlabel={\tiny $n$}, ylabel={\tiny time (s)},title={$\mu = 1.0$}]
            \addplot [select coords between index={6}{8}, color=TolBrightBlue, mark=o, mark size=3pt, thick] table[x=n, y=NCRPG_time] {\timeiter};
            \addplot [select coords between index={6}{8}, color=TolBrightCyan, mark=triangle, mark size=3pt, thick] table[x=n, y=NCRPG_bt_time] {\timeiter};
            \addplot [select coords between index={6}{8}, color=TolBrightRed, mark=star, mark size=3pt, thick] table[x=n, y=RPG_time] {\timeiter};
            \addplot [select coords between index={6}{8}, color=TolBrightYellow, mark=x, mark size=3pt, thick] table[x=n, y=ManPG_time] {\timeiter};

            \nextgroupplot[xlabel={\tiny $n$}, ylabel={\tiny sparsity}, ymin=0, ymax=1]
            \addplot [select coords between index={6}{8}, color=TolBrightBlue, mark=o, mark size=3pt, thick] table[x=n, y=NCRPG_sparsity] {\objspar};
            \addplot [select coords between index={6}{8}, color=TolBrightCyan, mark=triangle, mark size=3pt, thick] table[x=n, y=NCRPG_bt_sparsity] {\objspar};
            \addplot [select coords between index={6}{8}, color=TolBrightRed, mark=star, mark size=3pt, thick] table[x=n, y=RPG_sparsity] {\objspar};
            \addplot [select coords between index={6}{8}, color=TolBrightYellow, mark=x, mark size=3pt, thick] table[x=n, y=ManPG_sparsity] {\objspar};

            \legend{NCRPG, NCRPG bt, RPG, ManPG}

        \end{groupplot}

    \end{tikzpicture}
    \caption{Results for Sparse PCA on $\text{OB}(n,r)$ with $r = 5$, averaged over $10$ runs.
    The sparsity is defined as the proportion of entries with absolute value less than $10^{-8}$.}
    \label{fig:oblique-comparison-it-time-spar}
\end{figure}

\paragraph{\textbf{Discussion}}
In \Cref{fig:oblique-comparison-time}, we can see that both the norm of the gradient mapping and the error of the function value converge as expected for all four methods.
In fact, the local convergence rate seems to be linear.
Our method with constant stepsize is a little faster than the RPG with the same stepsize, whereas ManPG took significantly longer.
We note that our method also ran with larger stepsizes and was even faster then, but we chose the same stepsize for all methods.
The backtracking procedure significantly reduces the number of iterations and it was also faster in terms of CPU time.
See \Cref{fig:oblique-comparison-it-time-spar} for a summary.
We note that all methods converge to solutions with a relative difference in the cost function of less than $0.5\%$.
In particular, the solutions obtained by RPG and NCRPG with a constant step size have a relative difference of at most $10^{-10}$ and yield solutions with identical sparsity.
For larger penalty parameter $\mu$, sparsity is enhanced, as expected.
This comes at the cost of either orthogonality or similarity to the exact principal components.

\subsection{Experiments on the Real Grassmannian}\label{subsection:numerics-grassmannian}

The next experiment on the real Grassmannian is particular because here, we do not assume that the manifold is embedded.
While~\cite{HuangWei:2021:1} also do not need embeddedness in their theory, they at least require that the total space of a quotient manifold is embedded (see \cite[Assumption~3.6]{HuangWei:2021:1}), which, in the case of the Grassmannian, would be the Stiefel manifold.
Since they do not provide details on how their method would be implemented for the Grassmann manifold, we \emph{cannot} compare our algorithm to the RPG method in this case.

We aim to solve the optimization problem in \cref{eq:splitting} with
\begin{equation}
    g(p) = \frac{1}{2N} \sum_{j=1}^N \dist^2(p, q_j),
    \quad
    h(p) = \tau \dist(p, \bar{q}),
\end{equation}
where $\{q_1, \ldots, q_N\} \subset \cM$ is a set of $N=1000$ randomly generated data points, $\bar{q} \in \cM$ is another fixed point, $\tau \coloneq \frac{1}{2}$ is a weighting parameter, and $\cM = \grassmann{n}{r}$ is the Grassmann manifold of $r$-dimensional subspaces in $\bbR^n$.
The function $g$ represents a generalization of the sum of squared distances to data points, whose Riemannian gradient is $\grad g(p) = -\frac{1}{N} \sum_{j=1}^N \logarithm{p} q_j$.
The function $h$ acts as a regularization term encouraging proximity to the reference point $\bar{q}$, and its proximal mapping is given by
$
    \prox_{\lambda h}(p)
    =
    \geodesic<l>{p}{\bar q}(
        \min\{
            \lambda
            ,
            \dist(p, \bar q)
        \}
    )
    ,
$
where $\geodesic<l>{p}{\bar q}$ is the unit-speed geodesic that starts from $p$ in the direction of $\bar q$.
See, \eg, \cite[Proposition~1]{WeinmannDemaretStorath:2014:1} and, \eg, \cite{BendokatZimmermannAbsil:2024:1} for computational aspects and properties of the Grassmann manifold.

\paragraph{\textbf{Implementation Details}}
\footnote{The code is available at \url{https://juliamanifolds.github.io/ManoptExamples.jl/stable/examples/NCRPG-Grassmann/}}
The data points were generated by first choosing an anchor point $q_0 \in \cM$ uniformly at random.
Then, we generated $N+2$ normally distributed random tangent vectors to $q_0$ with a standard deviation of $\sigma = 1$, and mapped to $\cM$ using $\exponential{q_0}$.
Each tangent vector was first scaled down to be of length strictly less than $\frac{\pikmax}{2}$.
This ensures that the data points thus generated are contained in a uniquely geodesically convex set $\cU \subset \cM$.
Finally, we selected the first $N$ points as the data points $\{q_1, \ldots, q_N\}$ for the function $g$, and the last two points as the reference point $\bar{q}$ for $h$, and the initial point $p_0$ of the optimization respectively.
Since $\cU$ is a uniquely geodesically convex set containing all data points $q_j$, the function $g$ is $\zeta_{1, \kmin}(D)$-smooth over $\cU$, where $D = \diam(\cU)$, see, \eg, \cite[Lemma~2]{AlimisisBecigneulLucchiOrvieto:2020:1}.
Hence, because $\kmax = 2$, and $\kmin = 0$ on the Grassmannian, we can use the constant $\lipgrad = \zeta_{1, \kmin}(D) = 1$.

\begin{table}[tbp]
  \centering
  \begin{adjustbox}{max width=\textwidth}
    \begin{filecontents*}{data/grassmann/NCRPG-Grassmann-Comparisons.csv}
5,2,6,0.39143958300000004,90,0.8536782436604282,0.098821459,12,0.8536782436604244
10,4,24,0.5911909160000001,58,0.8583440635244626,0.17373629200000001,9,0.8583440635244599
50,10,400,2.438141041,36,0.868748624351462,0.7416927080000001,6,0.8687486243514602
100,20,1600,9.157380916000001,35,0.8717734180758447,6.141089166,6,0.8717734180758431
200,40,6400,36.830537208,34,0.8734261707269244,10.119820625000001,5,0.8734261707269231
\end{filecontents*}
\pgfplotstabletypeset[col sep = comma,
    columns={0,1,2,3,4,6,7},
    every head row/.style = {before row = \toprule
        \multicolumn{3}{c}{$\grassmann{n}{r}$} & \multicolumn{2}{c}{Constant stepsize} &\multicolumn{2}{c}{Backtracking}\\
        \cmidrule(lr){1-3}\cmidrule(lr){4-5}\cmidrule(lr){6-7},
    after row = \midrule},
    % Add bottomrule to the last row
    every last row/.style = {after row = \bottomrule},
    % Column formatting
    display columns/0/.style = {
        column name = $n$,
        string type,
        column type = {S[table-number-alignment = right, table-format = 5, table-alignment-mode = format]},
    },
    display columns/1/.style = {
        column name = $r$,
        string type,
        column type = {S[table-number-alignment = right, table-format = 5, table-alignment-mode = format]},
    },
    display columns/2/.style = {
        column name = Dimension,
        string type,
        column type = {S[table-number-alignment = right, table-format = 5, table-alignment-mode = format]},
	},
    display columns/3/.style = {
        column name = Time (sec.),
        string type,
        column type = {S[table-number-alignment = right, table-auto-round = true, scientific-notation = engineering, exponent-mode = scientific, table-format = 1.2e1]},
    },
    display columns/4/.style = {
        column name = Iter.,
        string type,
        column type = {S[table-number-alignment = right, table-auto-round = true, table-format = 5, table-alignment-mode = format]},
    },
    display columns/5/.style = {
        column name = Time (sec.),
        string type,
        column type = {S[table-number-alignment = right, table-auto-round = true, scientific-notation = engineering, exponent-mode = scientific, table-format = 1.2e1]},
    },
    display columns/6/.style = {
        column name = Iter.,
        string type,
        column type = {S[table-number-alignment = right, table-auto-round = true, table-format = 5, table-alignment-mode = format]},
    },
    multicolumn names,
    % mark by best times
    every row 0 column 5/.style={postproc cell content/.style={@cell content={\color{table-highlight-best}##1}}},
    every row 1 column 5/.style={postproc cell content/.style={@cell content={\color{table-highlight-best}##1}}},
    every row 2 column 5/.style={postproc cell content/.style={@cell content={\color{table-highlight-best}##1}}},
    every row 3 column 5/.style={postproc cell content/.style={@cell content={\color{table-highlight-best}##1}}},
    every row 4 column 5/.style={postproc cell content/.style={@cell content={\color{table-highlight-best}##1}}},
    % mark by least iterations
    every row 0 column 6/.style={postproc cell content/.style={@cell content={\color{table-highlight-best}##1}}},
    every row 1 column 6/.style={postproc cell content/.style={@cell content={\color{table-highlight-best}##1}}},
    every row 2 column 6/.style={postproc cell content/.style={@cell content={\color{table-highlight-best}##1}}},
    every row 3 column 6/.style={postproc cell content/.style={@cell content={\color{table-highlight-best}##1}}},
    every row 4 column 6/.style={postproc cell content/.style={@cell content={\color{table-highlight-best}##1}}},
    ]
    {data/grassmann/NCRPG-Grassmann-Comparisons.csv}
  \end{adjustbox}
  \vspace{1em}
	\caption{
        Comparison between constant and backtracked stepsize for NCRPG on $\grassmann{n}{r}$ with varying dimensions.
        Both methods converged to solutions with objective values that differ by approximately $10^{-12}$.
        Example from \cref{subsection:numerics-grassmannian}.
    }
	\label{table:NCRPG-CPPA-comparisons-time-iterations}
\end{table}
We compared \Cref{algorithm:NCRPG} with both a constant and a backtracked stepsize.
The constant stepsize was estimated on $\cU$ in accordance with the prescription in \cref{subsection:stepsize-discussion} with a choice of $\delta = 0.01$ for $\lambdadelta$ and $\zetadelta$.
The backtracking procedure was implemented as described in \cref{algorithm:backtracking} with an initial guess $s = 1$ and parameters $\beta = \frac{\zetadelta}{4}$, $\eta = 0.5$.
The results are summarized in \Cref{table:NCRPG-CPPA-comparisons-time-iterations}.
Both methods were terminated when the gradient mapping norm dropped below $10^{-7}$.
For every experiment, we report the manifold dimension and compare number of iterations and runtime of NCRPG with a constant stepsize and with backtracking.

\paragraph{\textbf{Discussion}}
The results show that the algorithm clearly works in the nonembedded case.
It can be seen that NCRPG with a backtracked stepsize outperforms the constant stepsize variant in all cases.

\subsection{Row-sparse Low-rank Matrix recovery}
\label{subsection:numerics-problem-setting-fixed-rank}
A common problem in signal recovery applications such as hyperspectral imaging \cite{GolbabaeeVandergheynst:2012:1} or blind deconvolution \cite{EisenmannKrahmerPfefferUschmajew:2022:1} is the task of recovering a signal $\bar X \in \bbR^{M \times N}$ from a set of $m$ observations that is simultaneously of low rank and admits a row-sparse structure.

Row-sparsity can be achieved by adding a $\ell_{1,2}$ penalty (\ie the sum of the $\ell_2$ norm of the rows) to the cost function.
Instead of minimizing the nuclear norm, as is often done, we assume the rank $r$ of the solution to be known. 
We get an optimization problem over the manifold of rank $r$ matrices $\cM_r$ \cite{EisenmannKrahmerPfefferUschmajew:2022:1}:
\begin{equation}
    \argmin_{X \in \cM_r}
    \,
    \norm{\bold A (X) - y}_2^2
    +
    \mu \norm{X}_{1,2}
    ,
\end{equation}
where $\bold A \colon \mathbb R^{M \times N} \rightarrow \mathbb R^m$ is a linear measurement operator that is often assumed to fulfill the {\em restricted isometry property (RIP)}.
The goal is to recover the (low-rank and row-sparse) signal $\bar X$ from as few measurements $y \coloneq \bold A (\bar X)$ as possible.

A convenient way to express a fixed rank matrix $X \in \cM_r$ is in the form of its compact SVD, \ie, $X = U \Sigma V^\top$, where $U^\top U = V^\top V = I_r$ and $\Sigma$ is the diagonal matrix containing the singular values of $X$.
Using this notation, the tangent space at $X$ can be expressed as \cite[(7.47)]{Boumal:2023:1}
\begin{equation}
    \begin{split}
    \tangentSpace{X}[\cM_r]
    =
    \setDef{
        U M V^\top
        +
        U_p V^\top
        +
        U V_p^\top
    }{
        &
        M \in \bbR^{r \times r},
        U_p \in \bbR^{M \times r},
        V_p \in \bbR^{N \times r}
        \\
        &
        \text{s.t. }
        U^\top U_p = 0,
        V^\top V_p = 0
    }
    .
    \end{split}
\end{equation}
The problem of low-rank and row-sparse matrix recovery does not fulfill all assumptions in \cref{assumption:nonconvex-assumptions}: The sectional curvature is bounded only locally and the sublevel sets are not compact since their intersection with the manifold is not closed.
However, locally, this should not be a problem, which is why we will test only local convergence.

Furthermore, to our knowledge, there is no closed form for geodesics on the fixed-rank manifold. Consequently, the exponential and logarithmic maps cannot be used for optimization over $\cM_r$. In our experiment, we will therefore use retractions (see \Cref{section:Retractions}) instead. As explained there, the convergence of a retraction based NCRPG remains an open problem. A detailed analysis of different retraction and inverse retraction methods, including their computation time and error compared to the exponential map, can be found in \cite{AbsilOseledets:2014:1}. In the experiment, we therefore use \Cref{algorithm:NCRPGretr}.

We will thus use the orthographic retraction which, for a given tangent vector expressed as ${H = UMV^\top + U_pV^\top + UV_p^\top}$, is given by 
\begin{equation}
    \retract{X}(H) =
        \left(
            U(\Sigma + M) + U_p
        \right)
        \left(
            \Sigma + M
        \right)^{-1}
        \left(
            \left(
                \Sigma + M
            \right)
            V^\top
            + V_p^\top
        \right)
    .
\end{equation}
For $Y \in \cM_r$, the corresponding inverse orthographic retraction is simply:
\begin{equation}
    \inverseRetract{X}(Y)
    =
    \proj{\tangentSpace{X}[\cM_r]}(Y - X)
    ,
\end{equation}
with the projection onto the tangent space at $X$ being given by
\begin{equation}
    \proj{\tangentSpace{X}[\cM_r]}(Z)
    =
    (I_M - U U^\top) Z V V^\top
    +
    U U^\top Z
    .
\end{equation}

We derive a closed-form solution for the proximal subproblem \cref{equation:retraction-objective} with $h = \mu \lVert \cdot \rVert_{1,2}$ for $\lambda > 0$ at $X$ using the retraction distance, i.e. finding a stationary point of
\begin{equation*}
    H^R(Y;X,\lambda)
    =
    \Bigl\|
        \proj{\tangentSpace{Y}[\cM_r]}\bigl(
            \mathrm{R}_{X}(-\lambda\grad g(X)) - Y
        \bigr)
    \Bigr\|_F^2
    +
    \mu \lVert Y \rVert_{1,2}
\end{equation*}
In particular, \cref{algorithm:NCRPGretr} can be viewed as an generalization of the Riemannian proximal gradient method proposed in \cite[Algorithm~3]{EisenmannKrahmerPfefferUschmajew:2022:1} for this specific problem.
In that method, a proximal update using the Euclidean distance is solved in each iteration $k$:
\begin{align*}\label{equation:prox-update-fixed-rank-eucl-dist}
    &\proxR[big]{\lambda \mu(X) \lVert \cdot \rVert_{1,2}}[\retractionSymbol](
        \mathrm{R}_{X}(-\lambda\grad g(X))
    )
    % uncomment for amspreprint
    \\
    &
    \qquad\qquad
    =
    \underset{Y \in \mathcal M_r}{\mathrm{argmin}}  \
    \frac{1}{2\lambda}
    \lVert
        Y
        -
        \mathrm{R}_{X}(-\lambda\grad g(X))
    \rVert_{\mathrm{F}}^{2}
    +
    \mu(X)
    \lVert
        Y
    \rVert_{1,2},
\end{align*}
where $\mu(X) = \tau^k \max_r (\Vert \mathrm{R}_{X}(-\lambda\grad g(X))_i  \Vert \colon i = 1, \ldots, M)$, \ie the norm of the $r$-th largest row of the current iterate, with $\tau \in (0, 1)$, is generally suggested.
We will see that
$\proxR[big]{\lambda \mu \lVert \cdot \rVert_{1,2}}[\retractionSymbol](
\mathrm{R}_{X}(-\lambda\grad g(X)))$
is also a viable update step for our method, and thus, the two methods coincide when $\mu(X) = \mu$ is chosen.
For $Z \in \cM_r$ , let $D^{\lambda \mu}[Z] \in \mathbb{R}^{M \times M}$ be the row-wise soft thresholding operator with
\begin{equation*}
  D^{\lambda\mu}[Z]_{ij}
  =
  \begin{cases}
    \frac{\lVert Z[i,:] \rVert_{2} - \lambda\mu}{\lVert Z[i,:] \rVert_{2}},
    & \qquad \text{if} \ i = j \text{ and } \lVert Z[i,:] \rVert_{2} > \lambda\mu, \\
    0, &\qquad \text{else}.
  \end{cases}
\end{equation*}
From~\cite[Proposition~3.4]{EisenmannKrahmerPfefferUschmajew:2022:1} we know that $D^{\lambda\mu}[Z]Z = \proxR[big]{\lambda \mu \lVert \cdot \rVert_{1,2}}[\retractionSymbol](Z)$.
Therefore, it holds that
\begin{equation*}
    0
    \in
    \mu
    \partial
    \lVert
        D^{\lambda \mu}[Z]Z
    \rVert_{1,2}
    +
    \frac{1}{\lambda}
    \proj{\tangentSpace{D^{\lambda \mu}[Z]Z}[\cM_r]}
    (
      D^{\lambda \mu}[Z]Z
      -
      Z
    )
    =
    \partial H^R(Y;X,\lambda)
    ,
\end{equation*}
where the equality is due to $Z = U\Sigma V^\top$ and $D^{\lambda \mu}[Z]Z = D^{\lambda \mu}[Z]U\Sigma V^\top$ lying in the tangent space of $D^{\lambda \mu}[Z]Z$ (because they have the same row-space $V$).
To see that this holds, one can compute the Euclidean differential of $f(Y) = \frac 1 2 \left\|
    \proj{\tangentSpace{Y}[\cM_r]}\left(Y-Z
    \right)
\right\|_F^2$ as follows.
\begin{align*}
    \d f(Y)[\xi]
    &=
    \langle
    \proj{\tangentSpace{Y}[\cM_r]}\left(Y-Z\right)
    ,
      \proj{\tangentSpace{Y}[\cM_r]} \xi
      +
    \d  (Y \mapsto \proj{\tangentSpace{Y}[\cM_r]})
    \left(Y-Z\right)
    [\xi]
    \rangle
    \\
    &=
    \langle
    \proj{\tangentSpace{Y}[\cM_r]}\left(Y-Z\right)
    ,
    \xi
    \rangle
    \\
    &
    \quad
    +
    \langle
      Y-Z
      ,
      \proj{\tangentSpace{Y}[\cM_r]}
      \d  (Y \mapsto \proj{\tangentSpace{Y}[\cM_r]})
      \circ
      \proj{\tangentSpace{Y}[\cM_r]} \left(Y-Z\right)
      [\xi]
    \rangle
    \\
    &=
    \langle
    \proj{\tangentSpace{Y}[\cM_r]}\left(Y-Z\right)
    ,
      \xi
    \rangle
    ,
\end{align*}
where the second step uses $Y-Z = \proj{\tangentSpace{Y}[\cM_r]}\left(Y-Z\right)$ and that the projection is self-adjoint.
The last step follows from the fact that for an operator $P$ with $P=P^2$ the identity $P (DP) P = 0$ holds.
Hence we get the closed form update step
\begin{equation*}
  \sequence{X}{k+1}
  =
  D^{\lambda \mu}[\mathrm{R}_{\sequence{X}{k}}(-\lambda\grad g(\sequence{X}{k}))] \mathrm{R}_{\sequence{X}{k}}(-\lambda\grad g(\sequence{X}{k})),
\end{equation*}
\ie, the prox step is simply a row-wise soft thresholding of the gradient step.

\paragraph{\textbf{Implementation Details}}
\label{subsection:numerics-implementation-fixed-rank}
\footnote{The code is available at \url{https://juliamanifolds.github.io/ManoptExamples.jl/stable/examples/NCRPG-Row-Sparse-Low-Rank/}}
We tested the local convergence behavior of NCRPG versus the Riemannian Alternating Direction Method of Multipliers (RADMM)~\cite{JiaxiangMaSrivastava:2022} for the problem described in \cref{subsection:numerics-problem-setting-fixed-rank}. This algorithm replaces the first prox-step (of the smooth part $g$) in the Euclidean ADMM~\cite{Beck:2017:1} with a gradient step on the manifold, whereas the second prox-step (of the nonsmooth part $h$) remains in the embedding space. As above, we used our closed form of the prox-step for $h$.

The measurement operator $\bold A \in \mathbb R^{m \times MN}$ was taken as a matrix with normally distributed entries in $\mathcal{N}(0, 1/\sqrt{m})$, such that $\bold A(X) = \bold A \mathrm{vec}(X) = y$.
We note that this choice fulfills the RIP condition with very high probability.

To show local convergence of our method, we first generated a random ground truth matrix ${X \in \bbR^{M\times N}}$ of rank $r$ with $s$ nonzero rows, which we perturb with a random matrix $Y \sim \mathcal{N}(0, 1/\sqrt{m})$ and truncate the rank.
The stepsize for NCRPG and RADMM with fixed stepsize was set to $\lambda = 0.25$.
Backtracking was done with an initial guess of $s = 0.5$ and parameters $\beta = 0.1$, $\eta = 0.5$ for \cref{algorithm:backtracking}.
The RADMM was implemented in \julia and the parameters $\rho = 0.1, \beta = 10^{-7}$ were obtained by a grid search.
We stopped NCRPG when the norm of the gradient mapping with respect to the orthographic retraction fell below $10^{-7}$ and the RADMM analogously when $\frac 1 \eta \distR(\sequence {X}{k}, \sequence {X}{k-1}) < 10^{-7}$.

\begin{centering}
\begin{table}[tbp]
    \centering
    \setlength{\tabcolsep}{5pt}
  \begin{adjustbox}{max width=\textwidth}
    \begin{filecontents*}{data/fixed_rank/results-fixed-rank-with-row-match-table.csv}
1,300,5.43886575,886,0.0,6.613728375,487,0.0,8.728315125,1249,5.012477416228689e-9
2,500,9.893382334,1028,1.598695500055152e-20,13.423571646,601,4.6484922277566535e-20,13.959405875,1307,6.344572250668174e-9
3,700,14.734491291,1120,2.876078181302642e-20,21.208973542,681,2.396877452569649e-20,19.7610784795,1400,6.663464743008381e-9
\end{filecontents*}
\pgfplotstabletypeset[col sep = comma,
    every head row/.style = {before row = \toprule
        \multicolumn{2}{c}{} & \multicolumn{6}{c}{NCRPG} & \multicolumn{3}{c}{RADMM}\\
        \multicolumn{2}{c}{Settings} & \multicolumn{3}{c}{Constant stepsize} &\multicolumn{3}{c}{Backtracking} &\\
        \cmidrule(lr){1-2}\cmidrule(lr){3-5}\cmidrule(lr){6-8}\cmidrule(lr){9-11},
    after row = \midrule},
    % Add lines to separate the experiments with different t
    every row no 2/.style = {after row = \midrule},
    every row no 5/.style = {after row = \midrule},
    % Add bottomrule to the last row
    every last row/.style = {after row = \bottomrule},
    % Column formatting
    display columns/0/.style = {
        column name = $r$,
        string type,
        column type = {S[table-number-alignment = right, table-format = 1, table-alignment-mode = format]},
    },
    display columns/1/.style = {
        column name = $m$,
        string type,
        column type = {S[table-number-alignment = right, table-format = 3, table-alignment-mode = format]},
    },
    display columns/2/.style = {
        column name = Time,
        string type,
        column type = {S[table-number-alignment = right, table-auto-round = true, table-format = 3.1, table-alignment-mode = format]},
    },
    display columns/3/.style = {
        column name = Iter.,
        string type,
        column type = {S[table-number-alignment = right, table-auto-round = true, table-format = 4, table-alignment-mode = format]},
    },
    display columns/4/.style = {
        column name = $\overline{\varepsilon_0}$,
        string type,
        column type = {S[table-number-alignment = right, table-auto-round = true, scientific-notation = engineering, exponent-mode = scientific, table-format = 1.1e1]},
    },
    display columns/5/.style = {
        column name = Time,
        string type,
        column type = {S[table-number-alignment = right, table-auto-round = true, table-format = 3.1, table-alignment-mode = format]},
    },
    display columns/6/.style = {
        column name = Iter.,
        string type,
        column type = {S[table-number-alignment = right, table-auto-round = true, table-format = 3, table-alignment-mode = format]},
    },
    display columns/7/.style = {
        column name = $\overline{\varepsilon_0}$,
        string type,
        column type = {S[table-number-alignment = right, table-auto-round = true, scientific-notation = engineering, exponent-mode = scientific, table-format = 1.1e1]},
    },
    display columns/8/.style = {
        column name = Time,
        string type,
        column type = {S[table-number-alignment = right, table-auto-round = true, table-format = 3.1, table-alignment-mode = format]},
    },
    display columns/9/.style = {
        column name = Iter.,
        string type,
        column type = {S[table-number-alignment = right, table-auto-round = true, table-format = 4, table-alignment-mode = format]},
    },
    display columns/10/.style = {
        column name = $\overline{\varepsilon_0}$,
        string type,
        column type = {S[table-number-alignment = right, table-auto-round = true, scientific-notation = engineering, exponent-mode = scientific, table-format = 1.1e1]},
    },
    multicolumn names,
    % mark by best times
    every row 0 column 2/.style={postproc cell content/.style={@cell content={\color{table-highlight-best}##1}}},
    every row 1 column 2/.style={postproc cell content/.style={@cell content={\color{table-highlight-best}##1}}},
    every row 2 column 2/.style={postproc cell content/.style={@cell content={\color{table-highlight-best}##1}}},
    % mark by least iterations
    every row 0 column 6/.style={postproc cell content/.style={@cell content={\color{table-highlight-best}##1}}},
    every row 1 column 6/.style={postproc cell content/.style={@cell content={\color{table-highlight-best}##1}}},
    every row 2 column 6/.style={postproc cell content/.style={@cell content={\color{table-highlight-best}##1}}},
    % mark by least mean error 
    every row 0 column 4/.style={postproc cell content/.style={@cell content={\color{table-highlight-best}##1}}},
    every row 1 column 4/.style={postproc cell content/.style={@cell content={\color{table-highlight-best}##1}}},
    every row 0 column 7/.style={postproc cell content/.style={@cell content={\color{table-highlight-best}##1}}},
    every row 2 column 7/.style={postproc cell content/.style={@cell content={\color{table-highlight-best}##1}}},
    ]
    {data/fixed_rank/results-fixed-rank-with-row-match-table.csv}
  \end{adjustbox}
  \vspace{1em}
	\caption{
      Comparison between NCRPG and the RADMM on $\mathbb \cM_r$ for the example from \cref{subsection:numerics-problem-setting-fixed-rank} with $M = 500$, $N = 100$, $s=10$ and $\mu = 10^{-4}$. $\overline{\varepsilon_0}$ denotes the mean of the row norms corresponding to the zero rows of the true solution.
      Times are in seconds.
  }
  \label{table:NCRPG-RADMM-comparisons-fixed-rank}
\end{table}
\end{centering}

\paragraph{\textbf{Discussion}}
Despite the problem not fulfilling all assumptions, we observe (local) convergence in all cases, see \cref{table:NCRPG-RADMM-comparisons-fixed-rank}.
The backtracking procedure is slower than NCRPG with a constant stepsize, because more costly evaluations of the objective are required.
All algorithms find the row-sparsity pattern of the ground truth, but the RADMM is less confident: the average norm of those rows is higher.

We note that the number of measurements $m$ is chosen such that the problem is underdetermined if either of the constraints (low rank, row-sparse) is considered alone.
Our results show that we recover a unique solution, indicating that, at least locally, the problem is well-posed.

\section{Conclusion \& Outlook}
\label{section:conclusion}
In this article, we introduced a nonconvex Riemannian proximal gradient method that works intrinsically on a Riemannian manifold.
We established state-of-the-art convergence results with only mild assumptions.
Convergence theory for an analogous method that uses retractions as an approximation of the exponential map remains an open problem.
In the experiments, we have demonstrated the efficiency of our method in many nonconvex and intrinsic cases that have so far been out of reach for nonsmooth Riemannian optimization.

In an upcoming work~\cite{BergmannJasaJohnPfeffer:2025:2}, we will consider the same method under stronger convexity assumptions, both on the functions $g$ and $h$, as well as on the manifolds, considering in particular Hadamard manifolds.
We will show that under these stronger assumptions, we can obtain improved convergence rates and even show convergence to stationary points.

{\subsection*{Acknowledgements}
\small
We would like to thank Lukas Klingbiel for his help with the low-rank row-sparse matrix recovery problem.
P.J.~was funded by the DFG – Projektnummer 448293816.}

% Insert the appendix.
\appendix

% Insert the bibliography.
\printbibliography

@article{ChenMaMan-ChoSoZhan:2020:1,
  TITLE = {Proximal Gradient Method for Nonsmooth Optimization over the Stiefel
           Manifold},
  VOLUME = {30},
  DOI = {10.1137/18m122457x},
  NUMBER = {1},
  JOURNAL = {SIAM Journal on Optimization},
  AUTHOR = {Chen, Shixiang and Ma, Shiqian and Man-Cho So, Anthony and Zhang,
            Tong},
  YEAR = {2020},
  PAGES = {210–239},
}

@article{ChoiChunJungYun:2024:1,
  TITLE = {On the linear convergence rate of Riemannian proximal gradient method
           },
  DOI = {10.1007/s11590-024-02129-6},
  JOURNAL = {Optimization Letters},
  AUTHOR = {Choi, Woocheol and Chun, Changbum and Jung, Yoon Mo and Yun,
            Sangwoon},
  YEAR = {2024},
}

@article{HuangWei:2021:1,
  TITLE = {Riemannian proximal gradient methods},
  VOLUME = {194},
  DOI = {10.1007/s10107-021-01632-3},
  NUMBER = {1–2},
  JOURNAL = {Mathematical Programming},
  AUTHOR = {Huang, Wen and Wei, Ke},
  YEAR = {2021},
  PAGES = {371–413},
}

@article{HuangWei:2023:1,
  TITLE = {An inexact Riemannian proximal gradient method},
  VOLUME = {85},
  DOI = {10.1007/s10589-023-00451-w},
  NUMBER = {1},
  JOURNAL = {Computational Optimization and Applications},
  AUTHOR = {Huang, Wen and Wei, Ke},
  YEAR = {2023},
  PAGES = {1–32},
}

@inproceedings{AlimisisBecigneulLucchiOrvieto:2020:1,
  AUTHOR = {Alimisis, Foivos and Orvieto, Antonio and B\'{e}cigneul, Gary and
            Lucchi, Aurelien},
  PAGES = {1297--1307},
  TITLE = {{A continuous-time perspective for modeling acceleration in
           Riemannian optimization}},
  YEAR = {2020},
  BOOKTITLE = {{Proceedings of the 23rd International Conference on Artificial
               Intelligence and Statistics}},
  EDITOR = {Chiappa, Silvia and Calandra, Roberto},
  PUBLISHER = {PMLR},
  URL = {https://proceedings.mlr.press/v108/alimisis20a.html},
  SERIES = {Proceedings of Machine Learning Research},
  VOLUME = {108},
}

@inproceedings{MartinezRubioPokutta:2023:1,
  title = {Accelerated Riemannian Optimization: Handling Constraints with a Prox
           to Bound Geometric Penalties},
  author = {Mart{\'i}nez-Rubio, David and Pokutta, Sebastian},
  booktitle = {Proceedings of Thirty Sixth Conference on Learning Theory},
  pages = {359--393},
  year = {2023},
  editor = {Neu, Gergely and Rosasco, Lorenzo},
  volume = {195},
  series = {Proceedings of Machine Learning Research},
  month = {12--15 Jul},
  publisher = {PMLR},
  url = {https://proceedings.mlr.press/v195/martinez-rubio23a.html},
}

@article{AlmeidaNetoOliveiraSouza2020,
  AUTHOR = {Almeida, Yldenilson Torres and {Cruz Neto}, Jo{\~{a}}o Xavier and
            Oliveira, Paulo Roberto and Souza, Jo{\~{a}}o Carlos de Oliveira},
  DATE = {2020},
  DOI = {10.1007/s10589-020-00173-3},
  JOURNALTITLE = {Computational Optimization and Applications},
  MONTH = {2},
  NUMBER = {3},
  PAGES = {649--673},
  PUBLISHER = {Springer Science and Business Media {LLC}},
  TITLE = {A modified proximal point method for {DC} functions on Hadamard
           manifolds},
  VOLUME = {76},
}

@article{SouzaOliveira2015,
  AUTHOR = {Souza, Jo{\~{a}}o Carlos Oliveira and Oliveira, Paulo Roberto},
  DATE = {2015},
  DOI = {10.1007/s10898-015-0282-7},
  JOURNALTITLE = {Journal of Global Optimization},
  MONTH = {2},
  NUMBER = {4},
  PAGES = {797--810},
  PUBLISHER = {Springer Science and Business Media {LLC}},
  TITLE = {A proximal point algorithm for {DC} fuctions on Hadamard manifolds},
  VOLUME = {63},
}

@article{BergmannFerreiraSantosSouza:2024,
  AUTHOR = {Bergmann, R. and Ferreira, O. P. and Santos, E. M. and Souza, J. C.
            O.},
  DOI = {10.1007/s10957-024-02392-8},
  EPRINT = {2112.05250},
  EPRINTTYPE = {arXiv},
  JOURNALTITLE = {Journal of Optimization Theory and Applications},
  TITLE = {The difference of convex algorithm on Hadamard manifolds},
  YEAR = {2024},
}

@online{SchielaHerzogBergmann:2024,
  AUTHOR = {Schiela, A. and Herzog, R. and Bergmann, R.},
  EPRINT = {2409.04492},
  EPRINTTYPE = {arxiv},
  TITLE = {Nonlinear Fenchel conjugates},
  YEAR = {2024},
}

@article{ZhaoYanZhu:2023,
  AUTHOR = {Zhao, Shimin and Yan, Tao and Zhu, Yuanguo},
  DOI = {10.1007/s10898-023-01326-4},
  JOURNAL = {Journal of Global Optimization},
  NUMBER = {4},
  PAGES = {1051–1076},
  TITLE = {Proximal gradient algorithm with trust region scheme on Riemannian
           manifold},
  VOLUME = {88},
  YEAR = {2023},
}

@article{HosseiniHuangYousefpour:2018:1,
  title = {Line search algorithms for locally Lipschitz functions on Riemannian
           manifolds},
  author = {Hosseini, Somayeh and Huang, Wen and Yousefpour, Rohollah},
  doi = {10.1137/16M1108145},
  journal = {SIAM Journal on Optimization},
  volume = {28},
  number = {1},
  pages = {596--619},
  year = {2018},
  publisher = {SIAM},
}

@article{FengHuangSongYingZeng:2021,
  AUTHOR = {Feng, Shuailing and Huang, Wen and Song, Lele and Ying, Shihui and
            Zeng, Tieyong},
  DOI = {10.1007/s11590-021-01822-0},
  JOURNAL = {Optimization Letters},
  NUMBER = {8},
  PAGES = {2277–2297},
  TITLE = {Proximal gradient method for nonconvex and nonsmooth optimization on
           Hadamard manifolds},
  VOLUME = {16},
  YEAR = {2021},
}

@online{BergmannJasaJohnPfeffer:2025:2,
  AUTHOR = {Bergmann, Ronny and Jasa, Hajg and John, Paula and Pfeffer, Max},
  TITLE = {The Intrinsic Riemannian Proximal Gradient Method for Convex
           Optimization},
  EPRINT = {2507.16055},
  EPRINTTYPE = {arxiv},
  TITLE = {The Intrinsic Riemannian Proximal Gradient Method for Convex
           Optimization},
  YEAR = {2025},
}

@inbook{KovnatskyGlashoffBronstein:2016,
  AUTHOR = {Kovnatsky, Artiom and Glashoff, Klaus and Bronstein, Michael M.},
  BOOKTITLE = {Computer Vision – ECCV 2016},
  DOI = {10.1007/978-3-319-46454-1_41},
  ISBN = {9783319464541},
  PAGES = {680–696},
  TITLE = {MADMM: A Generic Algorithm for Non-smooth Optimization on Manifolds},
  YEAR = {2016},
}

@misc{JiaxiangMaSrivastava:2022,
  AUTHOR = {Li, Jiaxiang and Ma, Shiqian and Srivastava, Tejes},
  COPYRIGHT = {arXiv.org perpetual, non-exclusive license},
  DOI = {10.48550/arxiv.2211.02163},
  TITLE = {A Riemannian ADMM},
  YEAR = {2022},
}

@article{GabayMercier:1976:1,
  AUTHOR = {Gabay, D. and Mercier, B.},
  DOI = {10.1016/0898-1221(76)90003-1},
  JOURNAL = {Computer and Mathematics with Applications},
  PAGES = {17--40},
  TITLE = {{A} dual algorithm for the solution of nonlinear variational problems
           via finite element approximations},
  VOLUME = {2},
  YEAR = {1976},
}

@article{EisenmannKrahmerPfefferUschmajew:2022:1,
  title = {Riemannian thresholding methods for row-sparse and low-rank matrix
           recovery},
  volume = {93},
  doi = {10.1007/s11075-022-01433-5},
  number = {2},
  journal = {Numerical Algorithms},
  author = {Eisenmann, Henrik and Krahmer, Felix and Pfeffer, Max and Uschmajew,
            André},
  year = {2022},
  pages = {669–693},
}

@article{BendokatZimmermannAbsil:2024:1,
  title = {A Grassmann manifold handbook: Basic geometry and computational
           aspects},
  author = {Bendokat, Thomas and Zimmermann, Ralf and Absil, P-A},
  journal = {Advances in Computational Mathematics},
  volume = {50},
  number = {1},
  pages = {6},
  year = {2024},
  publisher = {Springer},
}

@inproceedings{GolbabaeeVandergheynst:2012:1,
  title = {Hyperspectral image compressed sensing via low-rank and joint-sparse
           matrix recovery},
  author = {Golbabaee, Mohammad and Vandergheynst, Pierre},
  booktitle = {2012 IEEE International Conference on Acoustics, Speech and
               Signal Processing (ICASSP)},
  pages = {2741--2744},
  year = {2012},
  organization = {Ieee},
}

@software{BergmannJasa:2025:17277311,
  author = {Bergmann, Ronny and Jasa, Hajg},
  title = {ManoptExamples.jl},
  month = oct,
  year = 2025,
  publisher = {Zenodo},
  version = {v0.1.16},
  doi = {10.5281/zenodo.17277311},
}

@inproceedings{GenicotHuangTrendfilov:2015:1,
  author = {Genicot, Matthieu and Huang, Wen and Trendafilov, Nickolay T.},
  title = {Weakly Correlated Sparse Components with Nearly Orthonormal Loadings},
  booktitle = {Geometric Science of Information},
  year = {2015},
  publisher = {Springer International Publishing},
  pages = {484--490},
}

@article{LiYao:2012:1,
  author = {Li, Chong and Yao, Jen-Chih},
  title = {Variational Inequalities for Set-Valued Vector Fields on Riemannian
           Manifolds: Convexity of the Solution Set and the Proximal Point
           Algorithm},
  journal = {SIAM Journal on Control and Optimization},
  volume = {50},
  number = {4},
  pages = {2486-2514},
  year = {2012},
  doi = {10.1137/110834962},
}

@ARTICLE{AbsilOseledets:2014:1,
  AUTHOR = {Absil, P.-A. and Oseledets, I. V.},
  PUBLISHER = {Springer Science and Business Media LLC},
  DATE = {2014-12},
  DOI = {10.1007/s10589-014-9714-4},
  JOURNALTITLE = {Computational Optimization and Applications},
  NUMBER = {1},
  PAGES = {5--29},
  TITLE = {Low-rank retractions: a survey and new results},
  VOLUME = {62},
}

@INPROCEEDINGS{AlimisisOrvietoBecigneulLucchi:2021:1,
  AUTHOR = {Alimisis, Foivos and Orvieto, Antonio and Becigneul, Gary and Lucchi, Aurelien},
  EDITOR = {Banerjee, Arindam and Fukumizu, Kenji},
  PUBLISHER = {PMLR},
  URL = {https://proceedings.mlr.press/v130/alimisis21a.html},
  BOOKTITLE = {Proceedings of the 24th International Conference on Artificial Intelligence and Statistics},
  DATE = {2021},
  PAGES = {1351--1359},
  SERIES = {Proceedings of Machine Learning Research},
  TITLE = {Momentum improves optimization on Riemannian manifolds},
  VOLUME = {130},
}

@ARTICLE{Bacak:2014:1,
  AUTHOR = {Bačák, M.},
  DATE = {2014},
  DOI = {10.1137/140953393},
  JOURNALTITLE = {SIAM Journal on Optimization},
  NUMBER = {3},
  PAGES = {1542--1566},
  TITLE = {Computing medians and means in Hadamard spaces},
  VOLUME = {24},
}

@BOOK{Bacak:2014:2,
  AUTHOR = {Bačák, M.},
  LOCATION = {Berlin},
  PUBLISHER = {De Gruyter},
  DATE = {2014},
  DOI = {10.1515/9783110361629},
  SERIES = {De Gruyter Series in Nonlinear Analysis and Applications},
  TITLE = {Convex Analysis and Optimization in Hadamard Spaces},
  VOLUME = {22},
}

@BOOK{Beck:2017:1,
  AUTHOR = {Beck, A.},
  LOCATION = {Philadelphia, PA},
  PUBLISHER = {Society for Industrial and Applied Mathematics},
  DATE = {2017},
  DOI = {10.1137/1.9781611974997},
  TITLE = {First-Order Methods in Optimization},
}

@ARTICLE{BeckTeboulle:2009:1,
  AUTHOR = {Beck, Amir and Teboulle, Marc},
  DATE = {2009},
  DOI = {10.1137/080716542},
  JOURNALTITLE = {SIAM Journal on Imaging Sciences},
  NUMBER = {1},
  PAGES = {183--202},
  TITLE = {A fast iterative shrinkage-thresholding algorithm for linear inverse problems},
  VOLUME = {2},
}

@ARTICLE{BentoFerreiraOliveira:2015:1,
  AUTHOR = {Bento, Glaydston and Ferreira, Orizon and Oliveira, Paulo Roberto},
  DATE = {2015},
  DOI = {10.1080/02331934.2012.745531},
  JOURNALTITLE = {Optimization. A Journal of Mathematical Programming and Operations Research},
  NUMBER = {2},
  PAGES = {289--319},
  TITLE = {Proximal point method for a special class of nonconvex functions on Hadamard manifolds},
  VOLUME = {64},
}

@ONLINE{BergmannHerzogJasa:2024:1,
  AUTHOR = {Bergmann, Ronny and Herzog, Roland and Jasa, Hajg},
  DATE = {2024-02},
  EPRINT = {2402.13670},
  EPRINTTYPE = {arXiv},
  SCOOPINCLUDE = {yes},
  TITLE = {The Riemannian convex bundle method},
}

@ARTICLE{BergmannHerzogSilvaLouzeiroTenbrinckVidalNunez:2021:1,
  AUTHOR = {Bergmann, Ronny and Herzog, Roland and Silva Louzeiro, Maurício and Tenbrinck, Daniel and Vidal-Núñez, José},
  PUBLISHER = {Springer Science and Business Media LLC},
  DATE = {2021-01},
  DOI = {10.1007/s10208-020-09486-5},
  EPRINT = {1908.02022},
  EPRINTTYPE = {arXiv},
  JOURNALTITLE = {Foundations of Computational Mathematics},
  NUMBER = {6},
  PAGES = {1465--1504},
  SCOOPINCLUDE = {yes},
  TITLE = {Fenchel duality theory and a primal-dual algorithm on Riemannian manifolds},
  VOLUME = {21},
}

@ARTICLE{BergmannPerschSteidl:2016:1,
  AUTHOR = {Bergmann, Ronny and Persch, Johannes and Steidl, Gabriele},
  DATE = {2016},
  DOI = {10.1137/15M1052858},
  JOURNALTITLE = {SIAM Journal on Imaging Sciences},
  NUMBER = {4},
  PAGES = {901--937},
  TITLE = {A parallel Douglas Rachford algorithm for minimizing ROF-like functionals on images with values in symmetric Hadamard manifolds},
  VOLUME = {9},
}

@BOOK{Boumal:2023:1,
  AUTHOR = {Boumal, Nicolas},
  PUBLISHER = {Cambridge University Press},
  URL = {https://www.nicolasboumal.net/book},
  DATE = {2023-03},
  DOI = {10.1017/9781009166164},
  TITLE = {An Introduction to Optimization on Smooth Manifolds},
}

@BOOK{BridsonHaefliger:1999:1,
  AUTHOR = {Bridson, Martin R. and Haefliger, André},
  PUBLISHER = {Springer, Berlin},
  DATE = {1999},
  DOI = {10.1007/978-3-662-12494-9},
  SERIES = {Grundlehren der Mathematischen Wissenschaften [Fundamental Principles of Mathematical Sciences]},
  TITLE = {Metric Spaces of Non-Positive Curvature},
  VOLUME = {319},
}

@BOOK{DoCarmo:1992:1,
  AUTHOR = {do Carmo, Manfredo Perdigão},
  LOCATION = {Boston, MA},
  PUBLISHER = {Birkhäuser Boston, Inc.},
  DATE = {1992},
  SERIES = {Mathematics: Theory \& Applications},
  TITLE = {Riemannian Geometry},
}

@ARTICLE{FerreiraOliveira:1998:1,
  AUTHOR = {Ferreira, Orizon and Oliveira, Paulo Roberto},
  DATE = {1998},
  DOI = {10.1023/A:1022675100677},
  JOURNALTITLE = {Journal of Optimization Theory and Applications},
  NUMBER = {1},
  PAGES = {93--104},
  TITLE = {Subgradient algorithm on Riemannian manifolds},
  VOLUME = {97},
}

@ARTICLE{FerreiraOliveira:2002:1,
  AUTHOR = {Ferreira, Orizon and Oliveira, Paulo Roberto},
  DATE = {2002},
  DOI = {10.1080/02331930290019413},
  JOURNALTITLE = {Optimization. A Journal of Mathematical Programming and Operations Research},
  NUMBER = {2},
  PAGES = {257--270},
  TITLE = {Proximal point algorithm on Riemannian manifolds},
  VOLUME = {51},
}

@ARTICLE{HoseiniMonjeziNobakhtianPouryayevali:2021:1,
  AUTHOR = {Hoseini Monjezi, Najmeh and Nobakhtian, Soghra and Pouryayevali, Mohamad Reza},
  PUBLISHER = {Oxford University Press (OUP)},
  DATE = {2021-12},
  DOI = {10.1093/imanum/drab091},
  JOURNALTITLE = {IMA Journal of Numerical Analysis},
  TITLE = {A proximal bundle algorithm for nonsmooth optimization on Riemannian manifolds},
}

@ONLINE{LezcanoCasado:2020:1,
  AUTHOR = {Lezcano-Casado, Mario},
  DATE = {2020-08},
  EPRINT = {2008.02517},
  EPRINTTYPE = {arXiv},
  TITLE = {Curvature-dependant global convergence rates for optimization on manifolds of bounded geometry},
}

@ARTICLE{SilvaLouzeiroBergmannHerzog:2022:1,
  AUTHOR = {Silva Louzeiro, Maurício and Bergmann, Ronny and Herzog, Roland},
  PUBLISHER = {Society for Industrial \& Applied Mathematics (SIAM)},
  DATE = {2022-05},
  DOI = {10.1137/21m1400699},
  EPRINT = {2102.11155},
  EPRINTTYPE = {arXiv},
  JOURNALTITLE = {SIAM Journal on Optimization},
  NUMBER = {2},
  PAGES = {854--873},
  SCOOPINCLUDE = {yes},
  TITLE = {Fenchel duality and a separation theorem on Hadamard manifolds},
  VOLUME = {32},
}

@ARTICLE{WangLiYao:2016:1,
  AUTHOR = {Wang, Xiangmei and Li, Chong and Yao, Jen-Chih},
  PUBLISHER = {Springer Science and Business Media LLC},
  DATE = {2016-07},
  DOI = {10.1007/s10957-016-0979-x},
  JOURNALTITLE = {Journal of Optimization Theory and Applications},
  NUMBER = {3},
  PAGES = {783--803},
  TITLE = {On some basic results related to affine functions on Riemannian manifolds},
  VOLUME = {170},
}

@ARTICLE{WeinmannDemaretStorath:2014:1,
  AUTHOR = {Weinmann, Andreas and Demaret, Laurent and Storath, Martin},
  DATE = {2014},
  DOI = {10.1137/130951075},
  JOURNALTITLE = {SIAM Journal on Imaging Sciences},
  NUMBER = {4},
  PAGES = {2226--2257},
  TITLE = {Total variation regularization for manifold-valued data},
  VOLUME = {7},
}

@ARTICLE{YangZhangSong:2014:1,
  AUTHOR = {Yang, Wei Hong and Zhang, Lei-Hong and Song, Ruyi},
  DATE = {2014},
  JOURNALTITLE = {Pacific Journal of Optimization},
  NUMBER = {2},
  PAGES = {415--434},
  TITLE = {Optimality conditions for the nonlinear programming problems on Riemannian manifolds},
  VOLUME = {10},
}

@INPROCEEDINGS{ZhangSra:2016:1,
  AUTHOR = {Zhang, Hongyi and Sra, Suvrit},
  EDITOR = {Feldman, Vitaly and Rakhlin, Alexander and Shamir, Ohad},
  URL = {https://proceedings.mlr.press/v49/zhang16b.html},
  BOOKTITLE = {Proceedings of the 29th Annual Conference on Learning Theory},
  DATE = {2016-02},
  EPRINT = {1602.06053},
  EPRINTTYPE = {arXiv},
  PAGES = {1617--1638},
  SERIES = {Proceedings of Machine Learning Research},
  TITLE = {First-order methods for geodesically convex optimization},
  VOLUME = {49},
}

\end{document}